\documentclass[reqno,11pt]{amsart}
\usepackage{amsmath,amsfonts,amssymb,graphicx,amsthm,enumerate,url,tikz,bbm}
\usepackage[noadjust]{cite}
\usepackage{stmaryrd,paralist}
\usepackage{xifthen,xcolor,tikz,setspace,caption}
\usepackage[noadjust]{cite}
\usepackage[color,final]{showkeys} 

\colorlet{refkey}{orange!20}
\colorlet{labelkey}{blue!30}
\usepackage[colorlinks=true, pdfstartview=FitV, linkcolor=blue,
  citecolor=blue, urlcolor=blue,pagebackref=false]{hyperref}
\numberwithin{equation}{section}

\newtheorem{maintheorem}{Theorem}
\newtheorem{mainprop}[maintheorem]{Proposition}

\newtheorem{maincoro}[maintheorem]{Corollary}

\newtheorem{theorem}{Theorem}[section]
\newtheorem*{theorem*}{Theorem}
\newtheorem{lemma}[theorem]{Lemma}
\newtheorem{claim}[theorem]{Claim}
\newtheorem{fact}[theorem]{Fact}
\newtheorem{proposition}[theorem]{Proposition}
\newtheorem{corollary}[theorem]{Corollary}

\theoremstyle{definition}{

\newtheorem{remark}[theorem]{Remark}
\newtheorem*{remark*}{Remark}
}

\newcommand{\C}{\mathbb C}
\newcommand{\R}{\mathbb R}

\newcommand{\T}{\mathbb T}
\newcommand{\F}{\mathbb F}

\newcommand{\SRW}{\textsf{SRW}}
\newcommand{\NBRW}{\textsf{NBRW}}

\newcommand{\E}{\mathbb{E}}
\renewcommand{\P}{\mathbb{P}}
  
\renewcommand{\epsilon}{\varepsilon}

\newcommand{\cE}{{\mathcal{E}}}

\newcommand{\cX}{{\mathcal{X}}}
\newcommand{\cY}{{\mathcal{Y}}}

\newcommand{\cS}{{\mathcal{S}}}

\newcommand{\one}{\mathbbm{1}}

\newcommand{\Bin}{\operatorname{Bin}}

\newcommand{\dist}{\operatorname{dist}}
\newcommand{\Span}{\operatorname{Span}}
\newcommand{\vE}{{\vec{E}}}

\newcommand{\tmix}{t_\textsc{mix}}

\newcommand{\tv}{{\textsc{tv}}}

\newcommand{\diam}{\operatorname{diam}}
\newcommand{\Ltwo}{\mbox{\tiny $(L^2)$}}
\newcommand{\Lp}[1][p]{\mbox{\tiny $(L^p)$}}

\author[E. Lubetzky]{Eyal Lubetzky}
\address{E.\ Lubetzky\hfill\break
Courant Institute\\ New York University\\
251 Mercer St.\\ New York, NY~10012.}
\email{eyal@courant.nyu.edu}

\author[Y. Peres]{Yuval Peres}
\address{Y.\ Peres\hfill\break
Microsoft Research\\ One Microsoft Way\\ Redmond, WA 98052.}
\email{peres@microsoft.com}

\begin{document}

\title{Cutoff on all Ramanujan graphs}

\begin{abstract}
We show that on every Ramanujan graph $G$, the simple random walk exhibits cutoff: when $G$ has $n$ vertices and degree $d$,
the total-variation distance of the walk from the uniform distribution at time $t=\frac{d}{d-2}\log_{d-1} n + s\sqrt{\log n}$ is asymptotically $\P(Z > c\, s)$ where $Z$ is a standard normal variable and $c=c(d)$ is an explicit constant. Furthermore, for all $1 \leq p \leq \infty$, $d$-regular Ramanujan graphs minimize the asymptotic $L^p$-mixing time for \SRW\ among \emph{all} $d$-regular graphs.
Our proof also shows that, for every vertex $x$ in $G$ as above, its distance from $n-o(n)$ of the vertices is asymptotically $\log_{d-1} n$.
\end{abstract}
$\mbox{}$
\maketitle
\vspace{-.75cm}

\section{Introduction}\label{sec:intro}
A family of $d$-regular graphs $G_n$ with $d\geq 3$ fixed is called an \emph{expander}, following the works of Alon and Milman~\cite{AM85,Alon86} from the 1980's, if all nontrivial eigenvalues of the adjacency matrices are uniformly bounded away from~$d$.
Lubotzky, Phillips, and Sarnak~\cite{LPS88} defined a connected $d$-regular graph $G$ with $d\geq3$ to be \emph{Ramanujan} iff every eigenvalue $\lambda$ of its adjacency matrix is either $\pm d$ or satisfies $|\lambda|\leq2\sqrt{d-1}$. Such expanders, which in light of the Alon--Boppana Theorem~\cite{Nilli91} have an asymptotically optimal spectral gap, were first constructed, using deep number theoretic tools, in~\cite{LPS88} and independently by Margulis~\cite{Margulis88} (see also~\cite{Lubotzky2010,DSV03} and Fig.~\ref{fig:lps-ball4}). Due to their remarkable expansion properties, Ramanujan graphs have found numerous applications~(cf.~\cite{HLW06} and the references therein). However, after 25 years of study, the geometry of these objects is still mysterious, and in particular, determining the profile of distances between vertices in such a graph and the precise mixing time of simple random walk (\SRW) remained open.

Formally, letting $\|\mu-\nu\|_{\tv} =\sup_{A} [\mu(A)-\nu(A)]$ denote total-variation distance, the
 ($L^1$) mixing time of a finite Markov chain with transition kernel $P$ and stationarity distribution $\pi$ is defined as
\[ \tmix(\epsilon)=\min\{ t : D_{\tv}(t)\leq\epsilon\} \quad\mbox{ where }\quad D_{\tv}(t)=\max_{x}\|P^t(x,\cdot)-\pi\|_\tv\,.\]
A sequence of finite ergodic Markov chains is said to exhibit \emph{cutoff} if its total-variation distance from stationarity drops abruptly, over a period of time referred to as the \emph{cutoff window}, from near 1 to near 0; that is, there is cutoff iff $\tmix(\epsilon)=(1+o(1))\tmix(\epsilon')$ for any fixed $0<\epsilon,\epsilon'<1$.

Our main result shows that the Ramanujan assumption implies cutoff with $\tmix(\epsilon) =(\frac{d}{d-2}+o(1)) \log_{d-1}n$ and window $O(\sqrt{\log n})$.
\begin{maintheorem}\label{thm-srw}
On any sequence of $d$-regular non-bipartite Ramanujan graphs, \emph{\SRW} exhibits cutoff.
More precisely, let $G$ be such a graph on $n$ vertices and
\begin{equation*}
  t_\star(n) := \tfrac{d}{d-2}\log_{d-1} n\,.
\end{equation*}
Then for every fixed $s\in \R$ and every initial vertex $x$, the \emph{\SRW} satisfies
\begin{equation}
  \label{eq-dtv-srw}
  D_\tv\big(t_\star(n) + s \sqrt{\log_{d-1} n}\big) \to \P\left( Z > c_d\, s\right) \qquad\mbox{ as $n\to\infty$} \,,
\end{equation}
where $Z$ is a standard normal random variable and $c_d = \tfrac{(d-2)^{3/2}}{2\sqrt{d(d-1)}}$.
\end{maintheorem}

\begin{figure}[t]
\vspace{-0.45cm}
\centering
\raisebox{-0.6cm}{
\includegraphics[width=.54\textwidth]{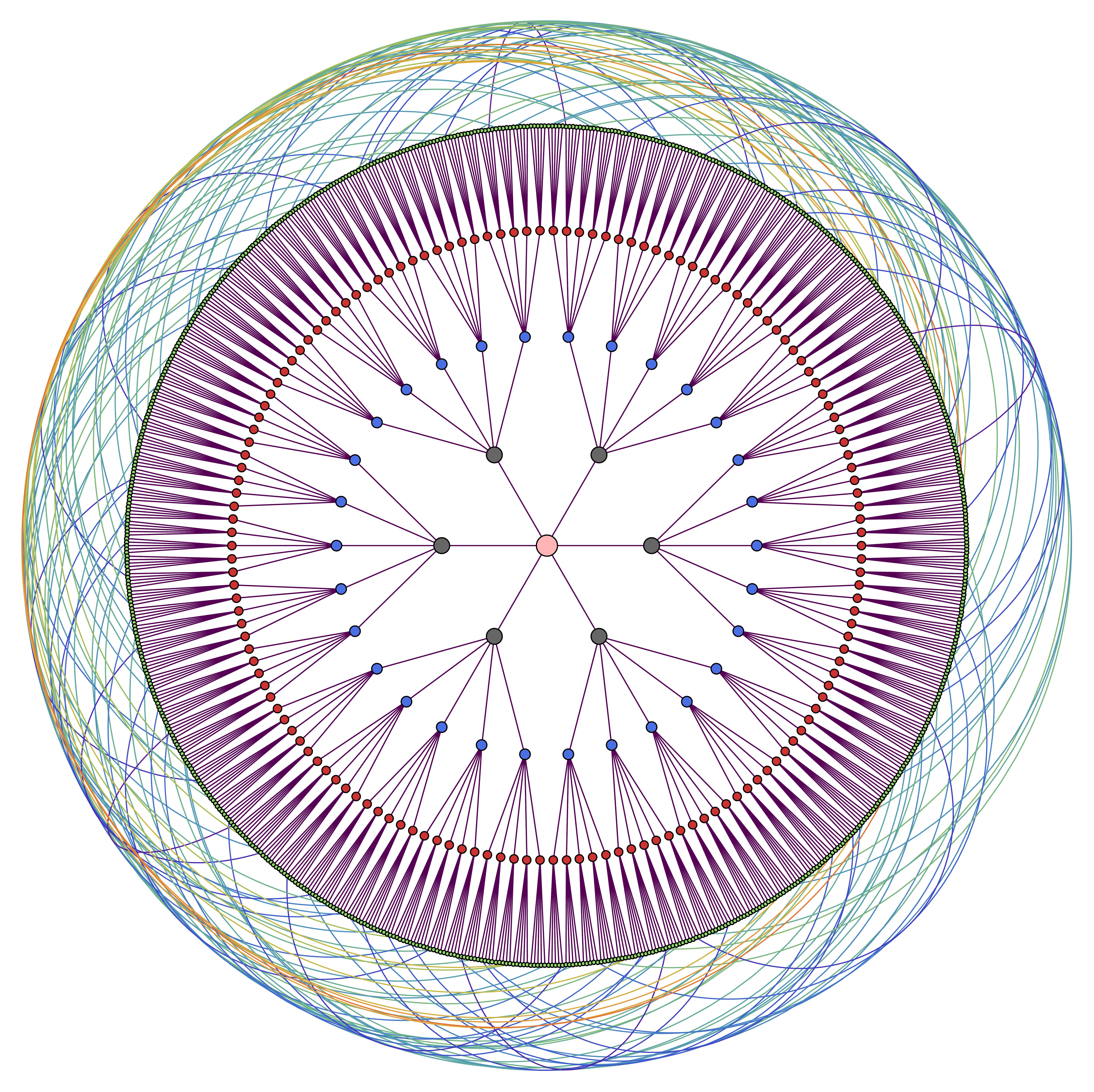}
}\hspace{0.25cm}
\includegraphics[width=.4\textwidth]{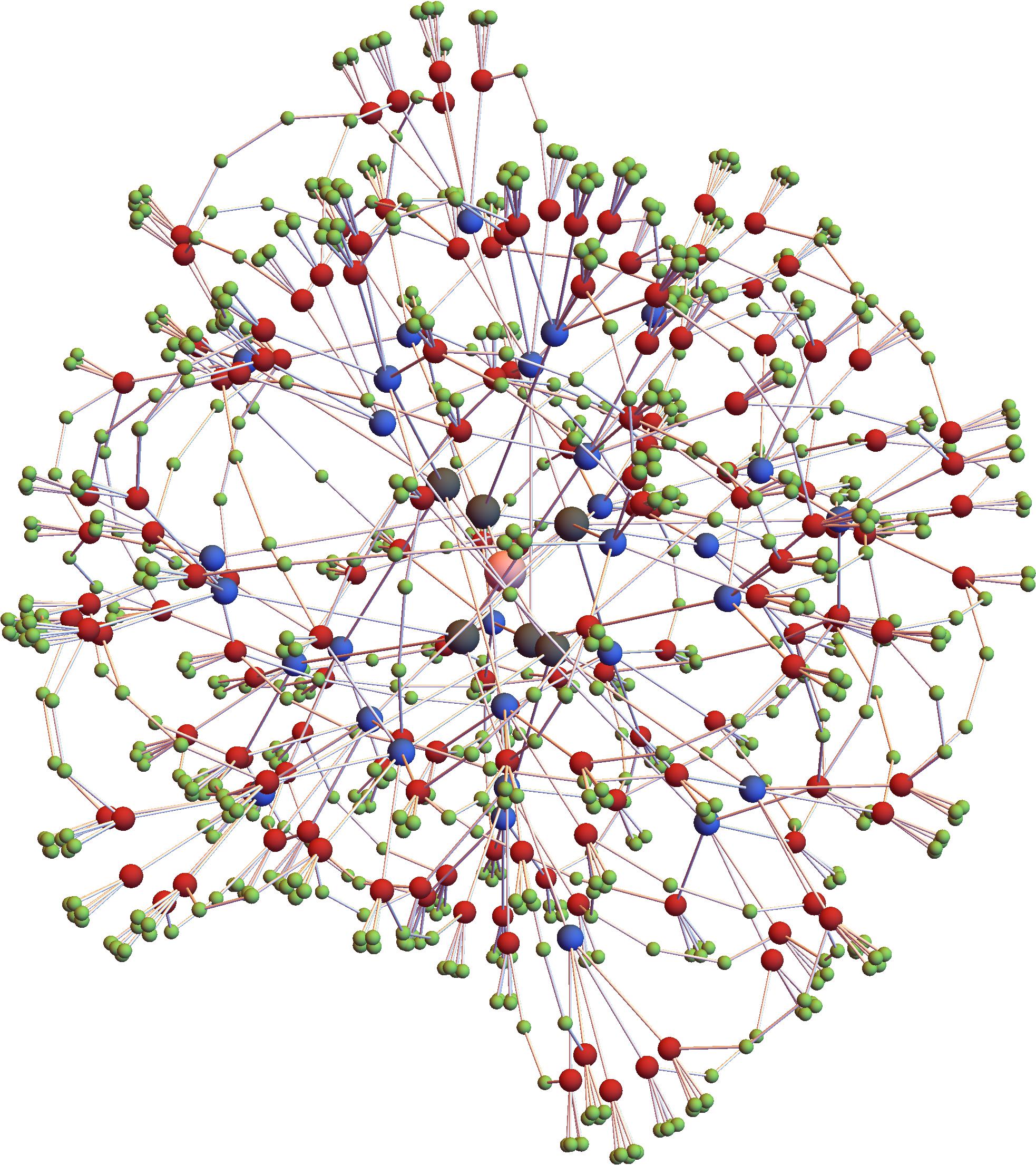}
\vspace{-0.48cm}
\caption{A ball of radius 4 in the Lubotzky--Phillips--Sarnak \mbox{6-regular} Ramanujan graph on $n=12180$ vertices via ${\rm PSL}(2,\F_{29})$.}
\label{fig:lps-ball4}
\vspace{-0.2cm}
\end{figure}

Consequently, we obtain that the profile of graph distances from every vertex $x$ in a $d$-regular Ramanujan graph $G$ concentrates on $\log_{d-1}n$ (the minimum possible value it can concentrate on in a $d$-regular graph).
\begin{maincoro}\label{cor-dist}
Let $G$ be a $d$-regular Ramanujan graph on $n$ vertices. Then for every vertex $x$ in $G$,
\[ \#\left\{ y : \left| \dist(x,y) - \log_{d-1} n \right| > 3\log_{d-1}\log n \right\} =o(n)\,,\]
and furthermore, for all except $o(n)$ vertices $y$ there is a nonbacktracking cycle\footnote{A nonbacktracking cycle is a sequence of adjacent vertices $v_0,\ldots,v_k$ such that $v_0=v_k$ and $v_i\neq v_{(i+2) \bmod k}$ for all $i$.}
 through $x,y$ of length at most $2\log_{d-1} n + 6\log_{d-1}\log n$.
\end{maincoro}

More can be said about high-girth Ramanujan graphs, {\em e.g.}, the bipartite LPS expanders whose girth is asymptotically $\frac43\log_{d-1} n$ (see, {\em e.g.},~\cite[\S7]{Lubotzky2010}).

\begin{maincoro}\label{cor-disjoint}
Let $G$ be a $d$-regular Ramanujan graph with $n$ vertices and girth $g$, and set $R=\lceil\log_{d-1} n + 5\log_{d-1}\log n\rceil$. For every $k\leq g-R$ and simple path $(x_i)_{i=1}^k $ in $G$, for all but a $o(1)$-fraction of simple paths $(y_i)_{i=1}^k $ in $G$, there are vertex-disjoint paths of length $R$ from $x_i$ to $y_i$ for all $i$.
\end{maincoro}

That Corollaries~\ref{cor-dist}--\ref{cor-disjoint} also cover bipartite Ramanujan graphs (recently shown to exist for every degree $d\geq 3$ in~\cite{MSS15}) follows form an extension of the proof of Theorem~\ref{thm-srw}
to the bipartite setting (see Corollary~\ref{cor-srw-bipartite}).
Moreover, it extends to the case where the graph $G$ is \emph{weakly} Ramanujan (see~\S\ref{sec:intro-weakly}).

\begin{figure}[t]
\vspace{-0.45cm}
\begin{center}
  \begin{tikzpicture}[font=\tiny]
    \newcommand{\hsep}{6.45cm}
    \node (plot1) at (\hsep,0) {
    \includegraphics[width=.47\textwidth]{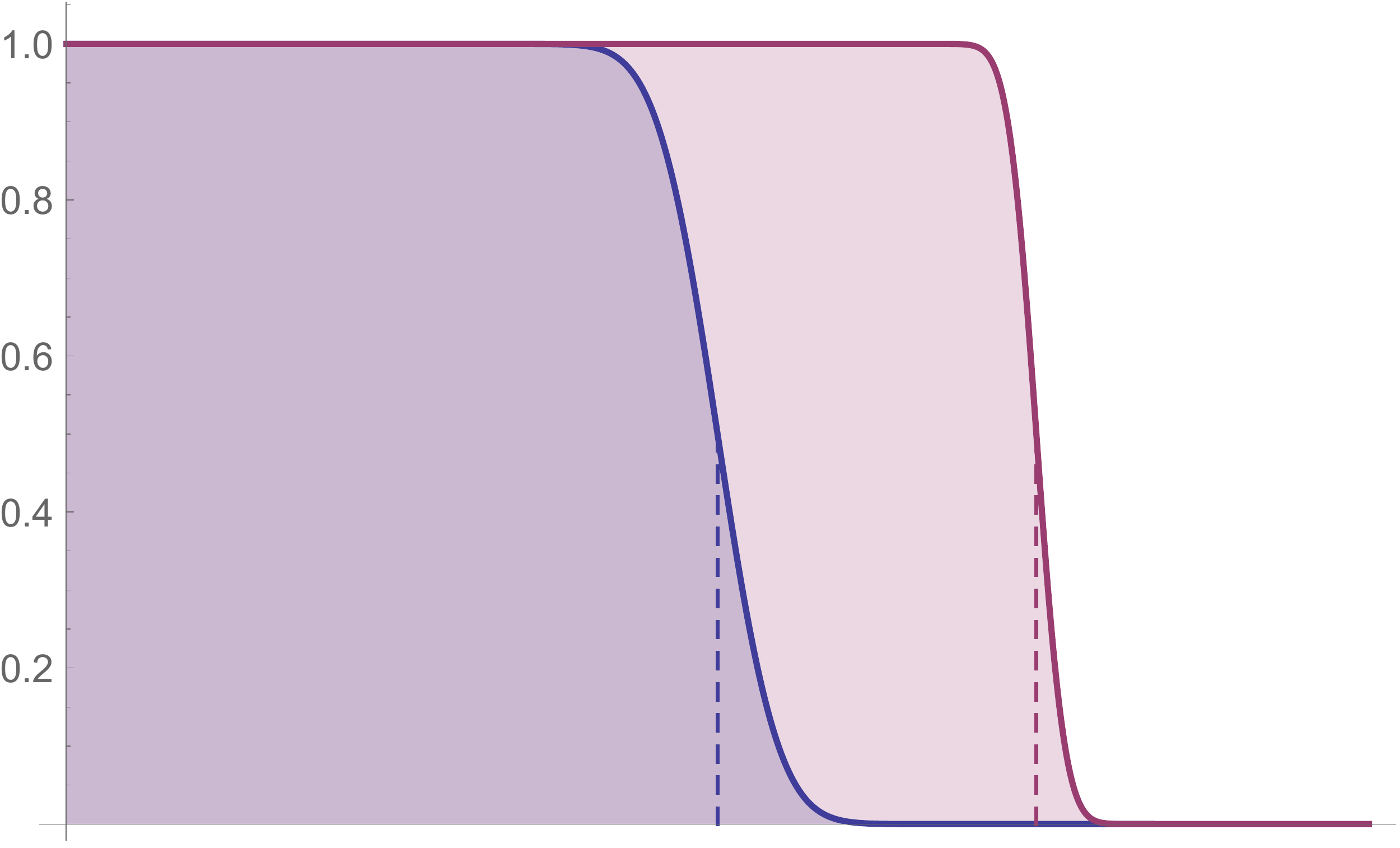}};
    \node (plot2) at (0,-2pt) {
    \includegraphics[width=.47\textwidth]{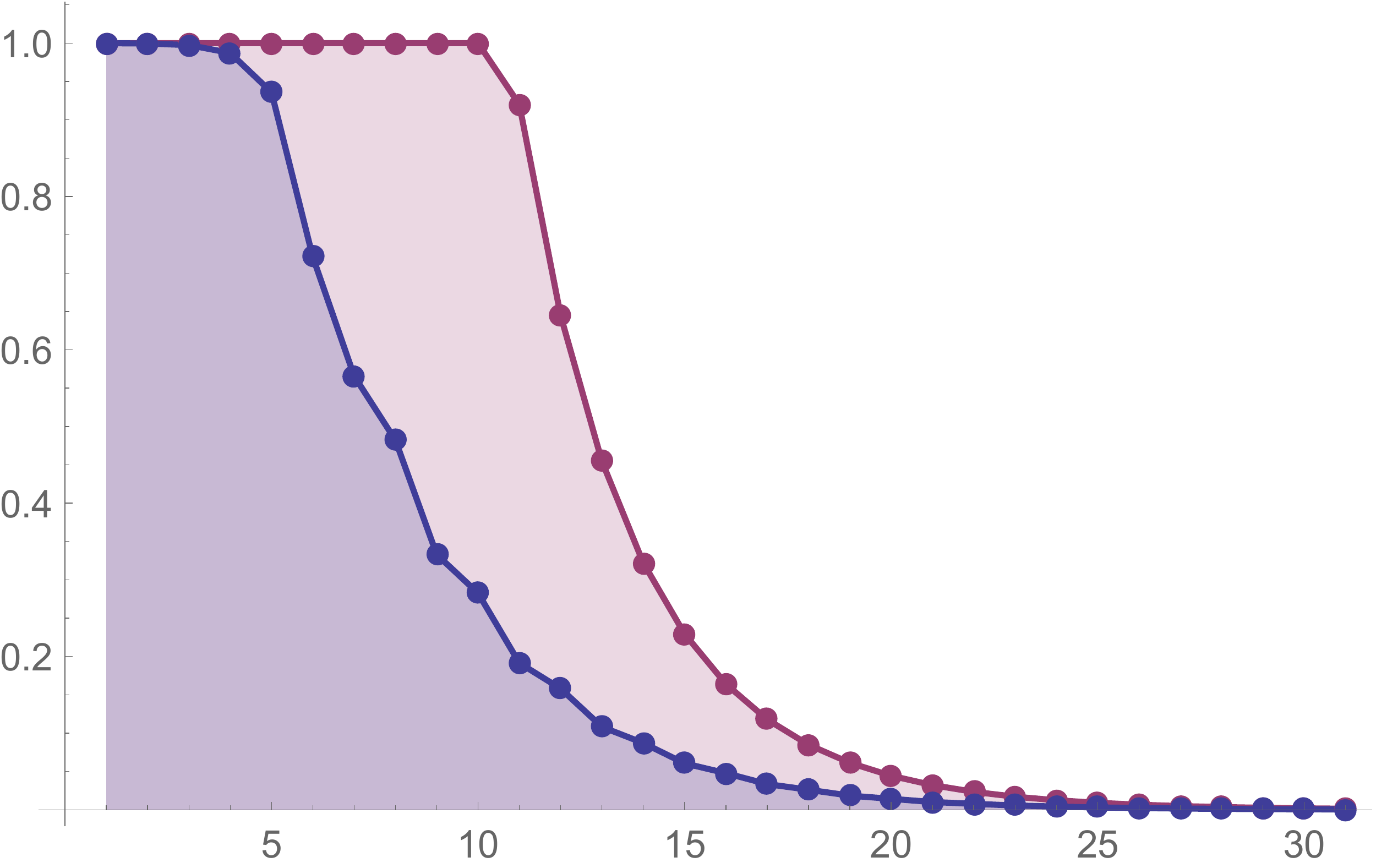}};
    \begin{scope}[shift={(plot1.south west)}]
    \node at (3,-0.05) {$\frac{d}{d-2}\log_{d-1}n$};
    \node at (4.75,-0.05) {$\frac12\log_{1/\rho}n$};
    \end{scope}
  \end{tikzpicture}
\end{center}
\vspace{-0.3cm}
\caption{Distance of \SRW\ from equilibrium in $L^1$ (blue) and $L^2$ (red, capped at 1). On left, the LPS graph on $\mathrm{PSL}(2,\F_{29})$ shown in Fig.~\ref{fig:lps-ball4}; on right, asymptotics via Theorem~\ref{thm-srw} and Proposition~\ref{prop:srw-lp}.}
\vspace{-0.6cm}
\label{fig:lps-l1-l2}
\end{figure}
\begin{figure}[t]
\begin{center}
  \begin{tikzpicture}[font=\tiny]
    \newcommand{\hsep}{6.45cm}
    \node (plot1) at (\hsep,0) {
    \includegraphics[width=.47\textwidth]{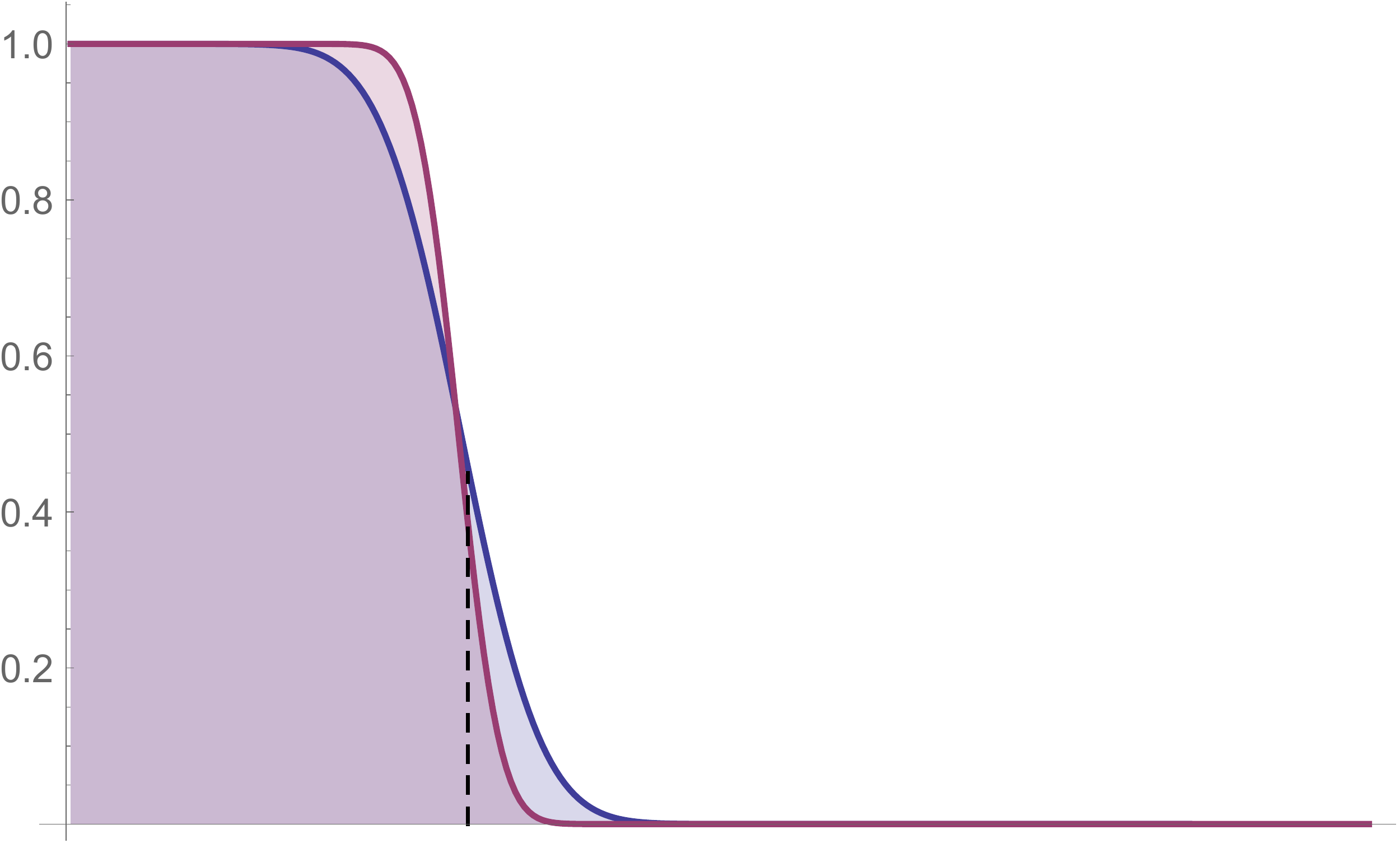}};
    \node (plot2) at (0,-2pt) {
    \includegraphics[width=.47\textwidth]{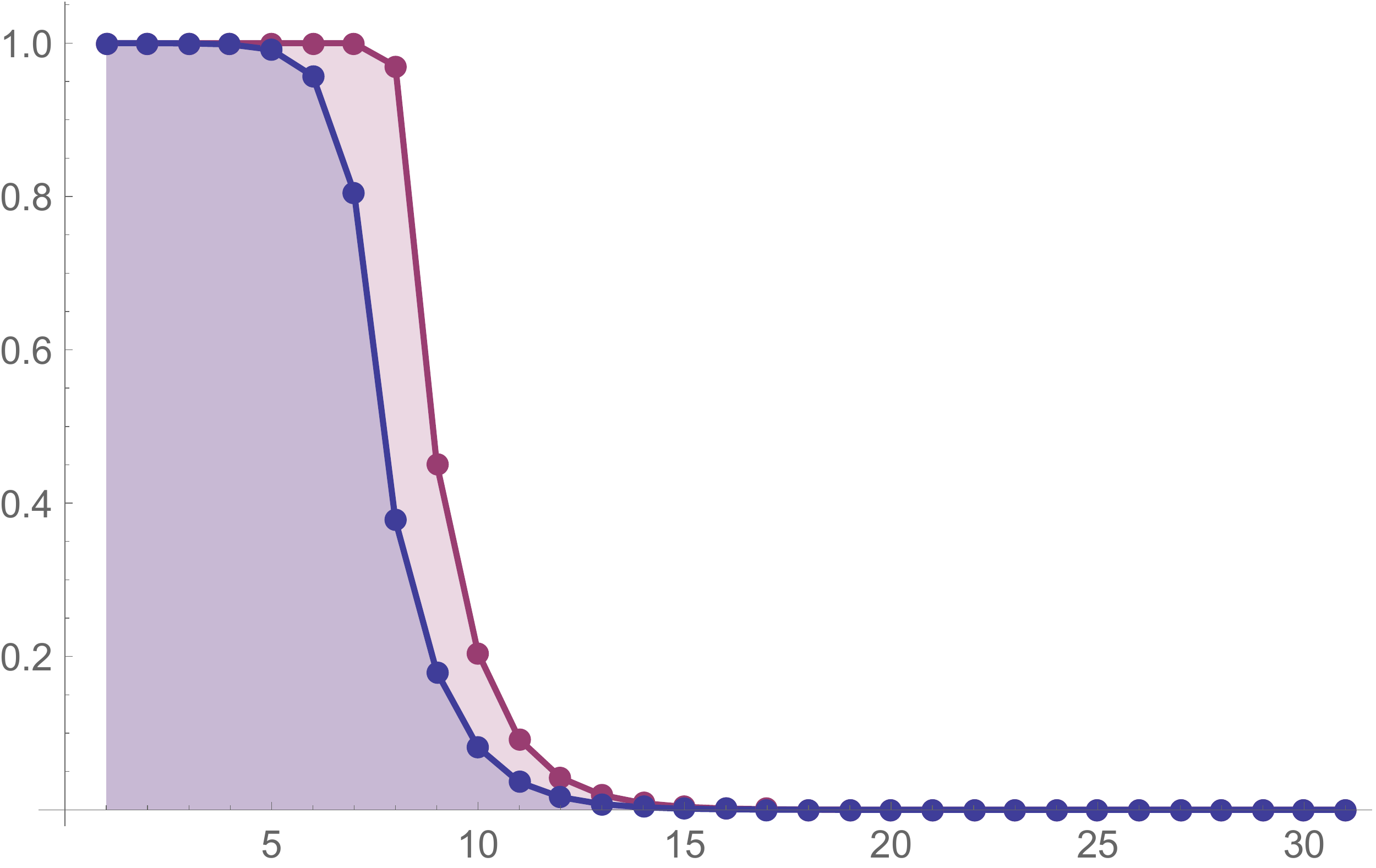}};
    \begin{scope}[shift={(plot1.south west)}]
    \node at (2.3,-0.05) {$\log_{d-1}n$};
    \end{scope}
  \end{tikzpicture}
\end{center}
\vspace{-0.35cm}
\caption{The analogue of Fig.~\ref{fig:lps-l1-l2} for the \NBRW\ (see~\S\ref{sec:intro-methods}).}
\vspace{-0.45cm}
\label{fig:lps-l1-l2-nbrw}
\end{figure}

\subsection{Background and related work}
The cutoff phenomenon was first identified in pioneering studies of Diaconis, Shahshahani and Aldous~\cite{Aldous,AD,DiaSha_81} in the early 1980's, and while believed to be widespread, rigorous examples where it was confirmed were scarce. In view of the canonical example where there is no cutoff---\SRW\ on a cycle---and the fact that a necessary condition for any reversible Markov chain to have cutoff is for its inverse spectral-gap to be negligible compared to its mixing time, the second author conjectured~\cite{PeresAIM04} in 2004 that on every transitive expander \SRW\ has cutoff.

Durrett~\cite[\S6]{DurrettRGW} conjectured in 2008 that the random walk should have cutoff on a uniformly chosen $d$-regular graph on $n$ vertices (typically a good expander) with probability tending to 1 as $n\to\infty$; indeed this is the case, as was verified by the first author and Sly~\cite{LS-gnd} in 2010. Subsequently, expanders without cutoff were constructed in~\cite{LSexp}, but these were highly asymmetric. The conjectured behavior of cutoff for all transitive expanders was reiterated in the latter works (see~\cite[Conjecture 6.1]{LS-gnd} and~\cite[\S3]{LSexp}), yet this was not verified nor refuted on any single example to date.

As a special case, Theorem~\ref{thm-srw} confirms cutoff on all transitive Ramanujan graphs---in particular for the Lubotzky--Phillips--Sarnak graphs (see~Fig.~\ref{fig:lps-l1-l2}).

The concentration of measure phenomenon in expanders, discovered by Alon and Milman~\cite{AM85}, implies that the distance from a prescribed vertex is concentrated up to an $O(1)$-window. Formally, for every sequence of expander graphs $G_n$ on $n$ vertices and vertex $x\in V(G_n)$ there exists a sequence $m_{n,x}$ and constants $a,C>0$ so that, for every $r>0$,
\begin{equation}
  \label{eq-dist-tail}
  \#\left\{ y\in V(G_n) : \left|\dist_{G_n}(x,y)-m_{n,x}\right|>r\right\} \leq C a^{-r} n\,.
\end{equation}
Corollary~\ref{cor-dist} shows that $m_{n,x} = \log_{d-1}n + O(\log\log n)$ for Ramanujan graphs.

As for the diameter, Alon and Milman (\cite[Theorem~2.7]{AM85}) showed that $\diam(G) \leq 2\sqrt{2d/(d-\lambda)}\log_2 n$ for every $d$-regular graph $G$ on $n$ vertices where all nontrivial eigenvalues are at most $\lambda$ in absolute value. This bound was improved to $\lceil \log_{d/\lambda}(n-1)\rceil$ by Chung~\cite[Theorem 1]{Chung89}, and then to $\big\lfloor \frac{\cosh^{-1}(n-1)}{\cosh^{-1}(d/\lambda)}\big\rfloor+1$ in~\cite{CFM94} using properties of $T_k(x)$, the Chebyshev polynomials of the first kind. Since $\cosh(\frac12\log(d-1))= d/(2\sqrt{d-1})$ for any $d$, this bound translates to $2\log_{d-1}n+O(1)$ for Ramanujan graphs, and remains the best known upper bound on the diameter of the LPS expanders (for which this was proved directly in~\cite{LPS88} via the polynomials $T_k(x)$ as later used in~\cite{CFM94}). Corollary~\ref{cor-dist} implies this asymptotically for every Ramanujan graph: as the distance from any vertex $x$ to most of the vertices is $(1+o(1))\log_{d-1}n$, the distance between any two vertices $x,y$ is at most $(2+o(1))\log_{d-1}n$. Moreover, one can deduce that for every two vertices $x,y$ and every integer $\ell\geq (2+o(1))\log_{d-1}n$, there exists a path of length \emph{exactly} $\ell$ between $x,y$.

A new impetus for understanding distances in Ramanujan graphs is due to their role as building blocks in quantum computing; see the influential letter by P.~Sarnak~\cite{Sarnak}. Some of Sarnak's ideas were developed further by his student N.T.~Sardari in an insightful paper~\cite{Sardari} posted to the arXiv a few months after the initial posting of the present paper.
For a certain infinite family of $(p+1)$-regular $n$-vertex Ramanujan graphs,    Sardari~\cite{Sardari} shows that the diameter is at least
$\lfloor \frac{4}{3} \log_p(n)\rfloor$ and also gives an alternative proof of the first part of Corollary~\ref{cor-dist}.

\subsection{Extensions}\label{sec:intro-weakly}
A sequence of connected $d$-regular graphs ($d\geq 3$ fixed) $G_n$ on $n$ vertices is called \emph{weakly Ramanujan} if, for some $\delta_n=o(1)$ as $n\to\infty$, every eigenvalue $\lambda$ of $G_n$ is either $\pm d$ or has $|\lambda|\leq 2\sqrt{d-1}+\delta_n$.

\begin{maintheorem}
  \label{thm-weakly}
On any sequence of $d$-regular non-bipartite weakly Ramanujan graphs, \emph{\SRW} exhibits cutoff.
More precisely, if $G_n$ is such a graph on $n$ vertices then for every initial vertex $x$, the \emph{\SRW}
has
\[ \tmix(\epsilon) = \big(\tfrac{d}{d-2}+o(1)\big)\log_{d-1} n\qquad\mbox{ for every fixed $0<\epsilon<1$}\,.\]
\end{maintheorem}

\begin{maincoro}\label{cor-dist-weakly}
Let $G_n$ be a $d$-regular weakly Ramanujan sequence of graphs on $n$ vertices. Then for every vertex $x$ in $G_n$,
\[ \#\Big\{ y : \Big|\frac{\dist(x,y)}{\log_{d-1}n} - 1\Big| > \epsilon\Big\} =o(n)\qquad\mbox{ for every fixed $\epsilon>0$}\,.\]
\end{maincoro}

\begin{remark*}
The weakly Ramanujan hypothesis in Theorem~\ref{thm-weakly} and Corollary~\ref{cor-dist-weakly} may be relaxed to allow some exceptional eigenvalues; for instance, we can allow $n^{o(1)}$ eigenvalues $\lambda$ to only satisfy $|\lambda| < d-\epsilon'$ for some $\epsilon'>0$ fixed.
\end{remark*}
By the result of Friedman~\cite{Friedman08} that a random (uniformly chosen) $d$-regular graph on $n$ vertices is typically weakly Ramanujan (as conjectured by Alon), Theorem~\ref{thm-weakly} then implies cutoff, re-deriving the above mentioned result of~\cite{LS-gnd}.

More generally, for two graphs $F$ and $G$, a \emph{covering map} $\phi:V(G)\to V(F)$ is a graph homomorphism that, for every $x \in V(G)$, induces a bijection between the
edges incident to $x$ and those incident to $\phi(x)$. If such a map exists, we say $G$ is a \emph{lift} (or a \emph{cover}) of $F$; a random $n$-lift of $F$ is a uniformly chosen lift out of all those with cover number $n$ ({\it i.e.}, $|\phi^{-1}(x)|=n$ for all $x$).

Friedman and Kohler~\cite{FK} recently proved (see also~\cite[Corollary~20]{Bordenave} by Bordenave) that for every fixed $d$-regular base graph $F$ and $\delta>0$, if $G$ is a random $n$-lift of $F$ then typically all of its ``new'' eigenvalues (those not inherited from $F$ via pullback) are at most $2\sqrt{d-1}+\delta$. By the remark above, Theorem~\ref{thm-weakly} and Corollary~\ref{cor-dist-weakly} apply here (for any fixed regular $F$).

\subsection{Cutoff in \texorpdfstring{$L^p$}{Lp}-distance}\label{sec:intro-lp}
Theorem~\ref{thm-srw} showed that Ramanujan graphs have an optimal $\tmix$ for \SRW: the total-variation distance from~\eqref{eq-dtv-srw} matches a lower bound valid for every $d$-regular graph on $n$ vertices (Fact~\ref{fact-srw-lower} in \S\ref{sec:srw}).
It turns out that Ramanujan graphs are extremal for $L^p$-mixing for \emph{all} $p\geq 1$.

For $1 \leq p\leq \infty$, the $L^p$-mixing time of a Markov chain with transition kernel $P$ from its stationary distribution $\pi$ is defined as
\[ \tmix^{\Lp}(\epsilon)=\min\{ t : D_{p}(t)\leq\epsilon\} \quad\mbox{ where }\quad D_{p}(t)=\max_{x}\Big\|\tfrac{P^t(x,\cdot)}\pi-1\Big\|_{L^p(\pi)}\]
(note that $p=1$ measures total-variation mixing since $D_1(t) = 2D_\tv(t)$, whereas the $L^2$-distance $D_2(t)$ is also known as the chi-square distance).
Chen and Saloff-Coste~\cite[Theorem 1.5]{ChenSaloffCoste08} showed that a lazy random walk on a family of expander graphs exhibits $L^p$-cutoff, at some unknown location, for all $p\in(1,\infty]$. (On the notable exception of $p=1$, it is said in~\cite{ChenSaloffCoste08} that there ``the question is more subtle and no good general answer is known.'')

The following theorem gives a lower bound on  for \SRW\ on a $d$-regular graph, asymptotically achieved by Ramanujan graphs for all $p\in (1,\infty]$.
\begin{mainprop}
  \label{prop:srw-lp}
  Fix $d\geq 3$ and let $\rho=2\sqrt{d-1}/d$. Then for all connected $d$-regular graphs $G$ on $n$ vertices and every fixed $\epsilon>0$, the \emph{\SRW} satisfies
  \begin{align}\label{eq-tmix-lp-lower-[2,inf]}
    \tmix^{{\Lp}}(\epsilon) &\geq \left\{\begin{array}
      {ll}
      c_{d,p} \log_{d-1} n - O(\log\log n)  & \mbox{ if $p\in(1,2]$}\\
     \tfrac{p-1}{p} \log_{1/\rho}n - O(\log\log n) & \mbox{ if $p\in[2,\infty]$}
    \end{array}\right.\,,
  \end{align}
  where $c_{d,p}= [2\beta-1+\frac{p}{p-1}H_{d-1}\big(\beta \;\|\; \frac{d-1}d\big)]^{-1}$ for $\beta=[1+(d-1)^{(p-2)/p}]^{-1}$ and $H_{b}(\beta\;\|\;\alpha) = \beta\log_b(\frac{\beta}\alpha)+(1-\beta)\log_b(\frac{1-\beta}{1-\alpha})$ is the relative entropy function.
  Furthermore, if $G$ is non-bipartite Ramanujan then, with the same notation,
  \begin{align}\label{eq-tmix-lp-rama-[2,inf]}
    \tmix^{{\Lp}}(\epsilon) &= \left\{\begin{array}
      {ll}
      c_{d,p}  \log_{d-1}n + O(\log\log n)  & \mbox{ if $p\in(1,2]$}\\
      \tfrac{p-1}{p} \log_{1/\rho}n + O(\log\log n) & \mbox{ if $p\in[2,\infty]$}
    \end{array}\right.\,.
  \end{align}
\end{mainprop}

\begin{figure}[t]
\vspace{-0.25cm}
\centering
\includegraphics[width=0.75\textwidth]{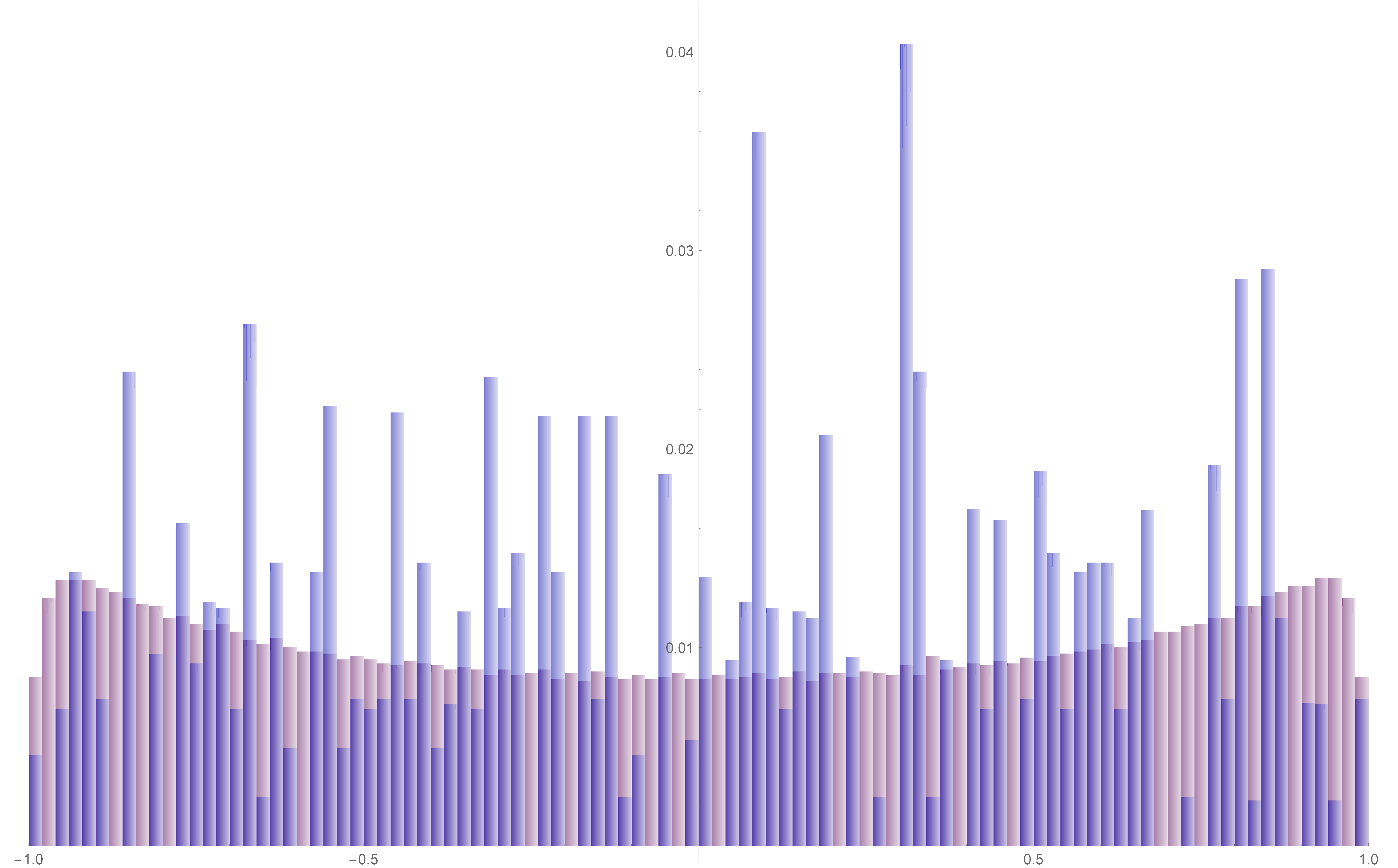}
\vspace{-0.2cm}
\caption{Normalized eigenvalues of Ramanujan graphs on $n\approx10^4$ vertices: the 6-regular LPS expanders on ${\rm PSL}(2,\F_{q})$ for $q=29$ (front; every nontrivial eigenvalue has multiplicity at least $\frac{q-1}2$) and a 1000-lift of the 3-regular Peterson graph (back).}
\label{fig:lps-pet-eig-A}
\end{figure}
\begin{figure}[t]
\vspace{-.4cm}
\centering
\raisebox{0.2cm}
{
\includegraphics[width=.48\textwidth]{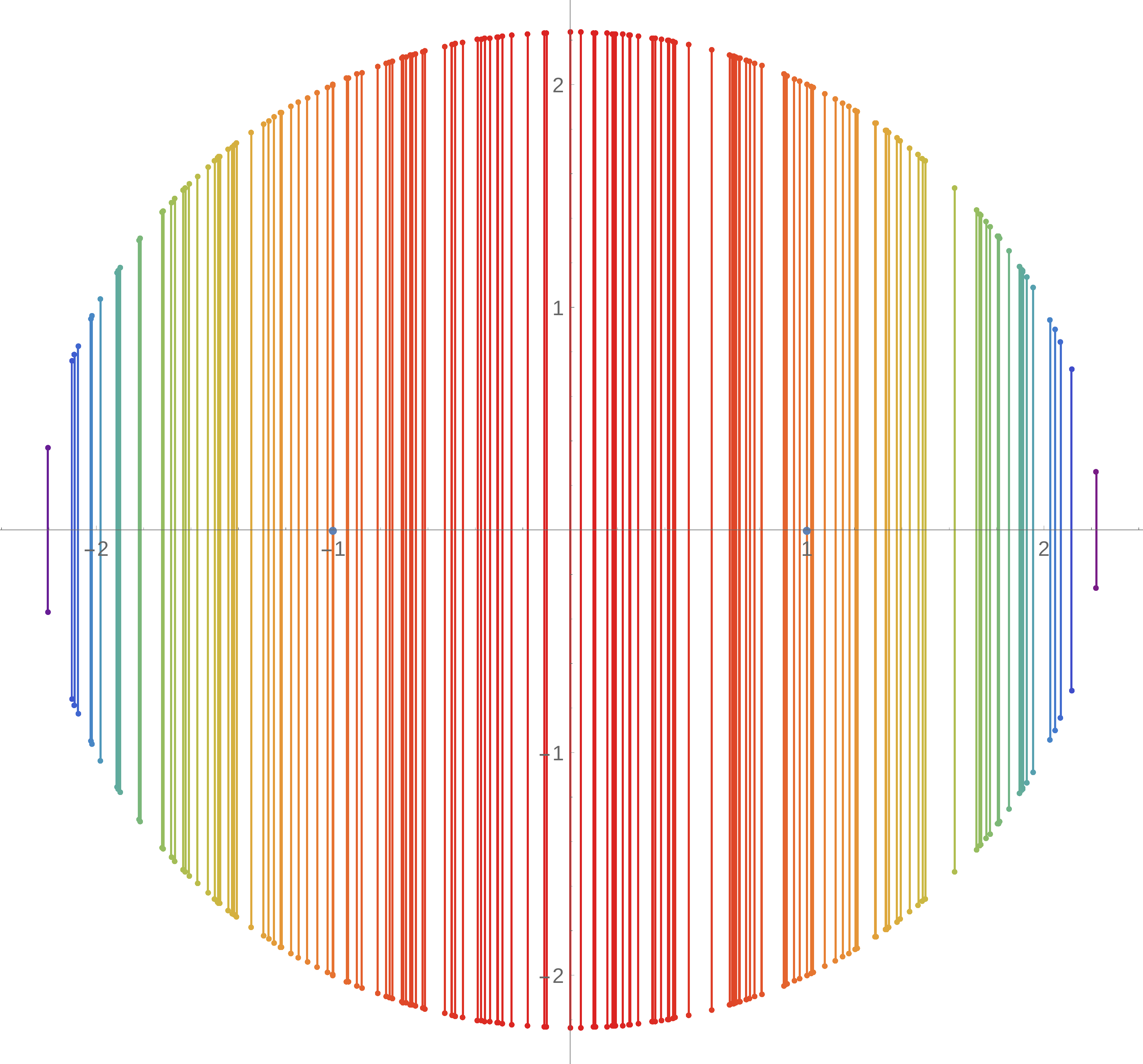}
}\hspace{-0.1cm}
\includegraphics[width=.48\textwidth]{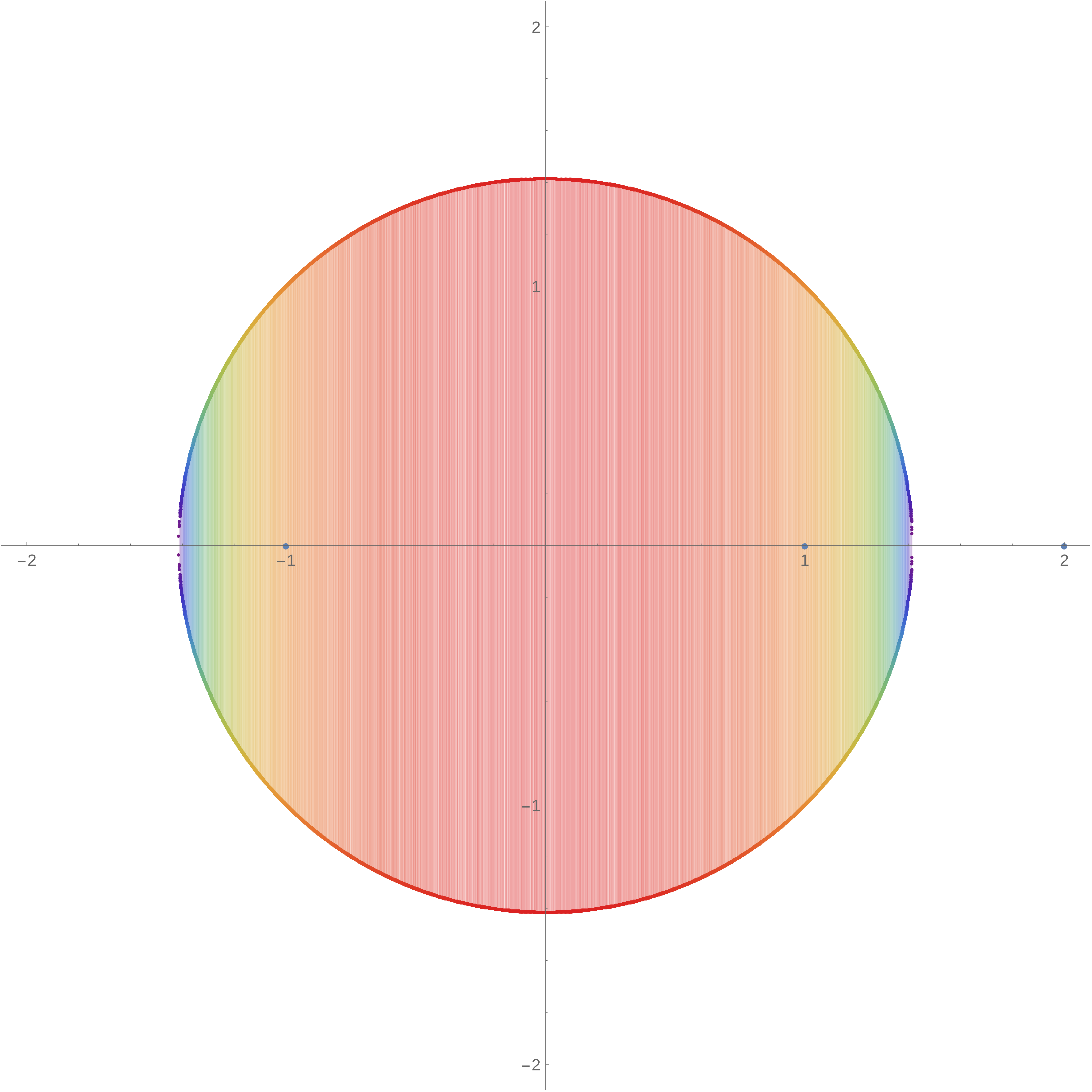}
\vspace{-0.2cm}
\caption{Eigenvalues of the nonbacktracking operator $B$ of the graphs from Fig.~\ref{fig:lps-pet-eig-A} (with the LPS expander on the left). Colors of chords between pairs of eigenvalues $\theta,\bar\theta$ depict the inner product of their corresponding eigenvectors $w,w'$ (blue near 1).
}
\label{fig:lps-pet-eig-B}
\vspace{-0.4cm}
\end{figure}

\subsection{Method of proof}\label{sec:intro-methods}
The natural route to exploit spectral details on the transition kernel $P$ for an upper bound on the $L^1$-distance from the stationary distribution $\pi$ is via the $L^2$-distance (see, \emph{e.g.},~\cite[Theorem 3.2]{HLW06}).
However, this fails to give the sought bound $\frac{d+o(1)}{d-2}\log_{d-1}n$  for the \SRW, as we see from Proposition~\ref{prop:srw-lp} that the \SRW\ on Ramanujan graphs exhibits an $L^2$-cutoff at $\frac12\log_{1/\rho}n > (1+\eta)\frac{d}{d-2}\log_{d-1}n $ for some $\eta(d)>0$ (see~\eqref{eq-d-infi}).

To remedy this, we turn to the nonbacktracking random walk (\NBRW), which moves from the directed edge $(x,y)$ to a uniformly chosen edge $(y,z)$ such that $z\neq x$.
In recent years, delicate spectral information on random graphs has been extracted by counting nonbacktracking paths; notably, this was essential in the proofs that random $d$-regular graphs and random lifts are weakly Ramanujan~\cite{Bordenave,Friedman08,FK}. Here we follow the reverse route, and use spectral information on the graph to control the nonbacktracking paths.

The known relation between the spectrum of $G$ and the spectrum of the nonbacktracking operator $B$ implies that if $G$ is Ramanujan, each of its nontrivial eigenvalues $\lambda$ is mapped to eigenvalues $\theta,\theta'\in\C$ of $B$ with modulus $\sqrt{d-1}$ (see Fig.~\ref{fig:lps-pet-eig-A}--\ref{fig:lps-pet-eig-B} showing this effect for two Ramanujan graphs with drastically different spectral features). For intuition, note that, had the operator $B$ been self-adjoint and transitive (it is neither), we would have gotten that the $L^2$-distance at time $t$ is $O( \sqrt{n} (d-1)^{-t/2})$, implying the correct upper bound of $(1+o(1))\log_{d-1} n$ for the \NBRW.

Fortunately, it turns out that, while not a normal operator, $B$ is unitarily similar to a matrix $\Lambda$ that is block-diagonal with  $n-1$ non-singleton blocks ($n-2$ if $G$ is bipartite), each of which has size $2\times 2$ (despite potential high multiplicities in the eigenvalues of $G$) and corresponds to an eigenvector pair $w,w'$ with matching eigenvalues $\theta,\bar\theta$. This description of $B$ appears in Proposition~\ref{prop:spectral} and may be of independent interest.

\subsection{Organization}
The rest of this paper is organized as follows. Section~\ref{sec:srw} describes the reduction of $L^1$-mixing for the \SRW\ to that of the \NBRW\ and establishes the optimality of the $L^p$-cutoff of \SRW\ on Ramanujan graphs for all $p> 1$ (Proposition~\ref{prop:srw-lp}). Section~\ref{sec:nbrw} studies the \NBRW, beginning in \S\ref{subsec:decomposition} with the aforementioned spectral decomposition and its properties (an exact computation of the off-diagonal entries is deferred to Proposition~\ref{prop-alphai} in~\S\ref{sec:off-diag-spectrum}). In~\S\ref{subsec:nonbipartite-nbrw} we give the proof of the non-bipartite case, which implies Theorem~\ref{thm-srw} and Corollary~\ref{cor-dist}; and~\S\ref{subsec:extensions} includes the proofs of the extensions to bipartite and weakly Ramanujan graphs, which imply Theorem~\ref{thm-weakly} and Corollary~\ref{cor-dist-weakly}.

\section{Simple random walk}\label{sec:srw}
\subsection{Reduction to NBRW}\label{sec:reduction}
As described in~\cite{LS-gnd} (see~\S2.3 and \S5.2 there), cutoff for \SRW\ can be reduced to cutoff for the \NBRW\ as follows.
Let $G=(V,E)$ be a $d$-regular graph and let $\T_d$ be the infinite regular tree rooted at $\xi$, the universal cover of $G$.
For a given vertex $x\in V$, consider a cover map $\phi:\T_d\to V$ with $\phi(\xi)=x$, and observe that if $(\cX_t)$ is \SRW\ on $\T_d$ started at $\xi$, then $X_t = \phi(\cX_t)$ is \SRW\ on $G$ started at $x$.
(This was also used in the proof of the Alon--Boppana Theorem given in~\cite[Proposition~4.2]{LPS88}.)
Similarly, if $\cY_t$  is \NBRW\ on $\T_d$ started at $(\xi,\sigma)\in \vE(\T_d)$, and we write $\cY_t= (\cY_t',\cY_t'')$ to denote its endpoint vertices,
then $Y_t = (Y_t',Y_t'')$ given by $Y_t'=\phi(\cY_t')$ and $Y_t''=\phi(\cY_t'')$ is \NBRW\ on $G$ started at $(x,\phi(\sigma))$. By symmetry, if
\[ \cE_{t,\ell}:= \left\{ \dist(\xi,\cX_t)=\ell\right\}\,,\]
then the conditional distribution of $\cX_t$ given $\cE_{t,\ell}$ is uniform over the vertices at distance $\ell$ from $\xi$ in $\T_d$. Therefore,
\[ \P_x\left(X_t \in \cdot \mid \cE_{t,\ell}\right) = \frac1d\sum_{\sigma: \xi\sigma\in E(\T_d)} \P_{(x,\phi(\sigma))}\left(Y''_{\ell-1} \in \cdot\right)\,.\]
As a projection can only decrease total-variation distance, letting $\ell= \tmix(\epsilon)$ for the \NBRW\ on $G$ and $\pi$ be the uniform distribution over $V(G)$, we get
\[\left\| \P_x\left(X_t\in\cdot\right) - \pi\right\|_{\tv} \leq \epsilon+\P\left(\dist(\xi,\cX_t) < \ell\right) \,,\]
and in particular, taking a maximum over $x$ shows that the \SRW\ on $G$ has
\begin{equation}
  \label{eq-Xt-dtv}
D_\tv (t) \leq \epsilon+\P(\dist(\xi,\cX_t)<\ell)\,.
\end{equation}
Finally, since \SRW\ on $\T_d$ ($d\geq 3$) is transient, $\cX_t$ returns to $\xi$ only a finite number of times almost surely.  If $\cX_t \neq \xi$ then
$\dist(\cX_{t+1},\xi) - \dist(\cX_{t},\xi)$ is equal to $-1$ with probability $1/d$ and $+1$ otherwise.
Therefore, by the CLT,
\begin{equation}\label{eq-Xt-tree-N}
\frac{\dist(\cX_{t},\xi) - ((d-2)/d)t}{(2\sqrt{d-1}/d)\sqrt{ t}} \Rightarrow \mathcal{N}(0,1)\,.
\end{equation}
Thus, if $\ell\to\infty$ then by~\eqref{eq-Xt-dtv}, for every fixed $s \geq 0$ the \SRW\ on $G$ satisfies
\begin{equation}
  \label{eq-dtv-upper}
  \limsup_{n\to\infty}D_\tv\left(\tfrac{d}{d-2}\ell + s\sqrt{\ell}\right) \leq \epsilon + \P\left(Z > c_d \, s\right) \,,
\end{equation}
where $Z$ is a standard normal random variable and $c_d = \frac{(d-2)^{3/2}}{2\sqrt{d(d-1)}}$.

Conversely, the number of vertices at distance $\ell$ from a given vertex $x\in V$ is at most $d(d-1)^\ell$.
So, on the event $\dist(\cX_t,\xi)< \log_{d-1}(\epsilon n/d)$, the \SRW\ $X_t$ is confined to a set of at most $\epsilon n$ vertices of $G$, thus its total-variation distance from $\pi$ is at least $1-\epsilon$. Altogether,~\eqref{eq-Xt-tree-N} implies the following.
\begin{fact}\label{fact-srw-lower}
For every $d$-regular graph on $n$ vertices with $d\geq 3$ fixed, and every fixed $s,\epsilon>0$, the \emph{\SRW} on $G$ satisfies
\begin{equation}
  \label{eq-dtv-lower}
   \liminf_{n\to\infty} D_\tv\left(t - s \sqrt{\log_{d-1} n}\right) \geq  1-\epsilon - \P\left( Z > c_d\, s\right)
\end{equation}
at $t=\frac{d}{d-2}\log_{d-1} (\epsilon n/d)$, where $c_d = \frac{(d-2)^{3/2}}{2\sqrt{d(d-1)}}$ and $Z\sim \mathcal{N}(0,1)$.
\end{fact}

Comparing the two bounds~\eqref{eq-dtv-upper}--\eqref{eq-dtv-lower} with the desired estimate~\eqref{eq-dtv-srw} for the \SRW\ in Theorem~\ref{thm-srw}, we see that the latter will follow if we show that the \NBRW\ has cutoff at time $\log_{d-1} n + o(\sqrt{\log n})$ with window $o(\sqrt{\log n})$. This will be achieved in~\S\ref{sec:nbrw} via a spectral analysis of the nonbacktracking walk.

\subsection{Optimal \texorpdfstring{$L^p$}{Lp}-mixing on Ramanujan graphs}
We begin with the special case $p=2$ of Proposition~\ref{prop:srw-lp}.

\begin{figure}[t]
\begin{center}
  \begin{tikzpicture}[font=\tiny]
    \node (plot1) at (0,0) {
    \includegraphics[width=0.85\textwidth]{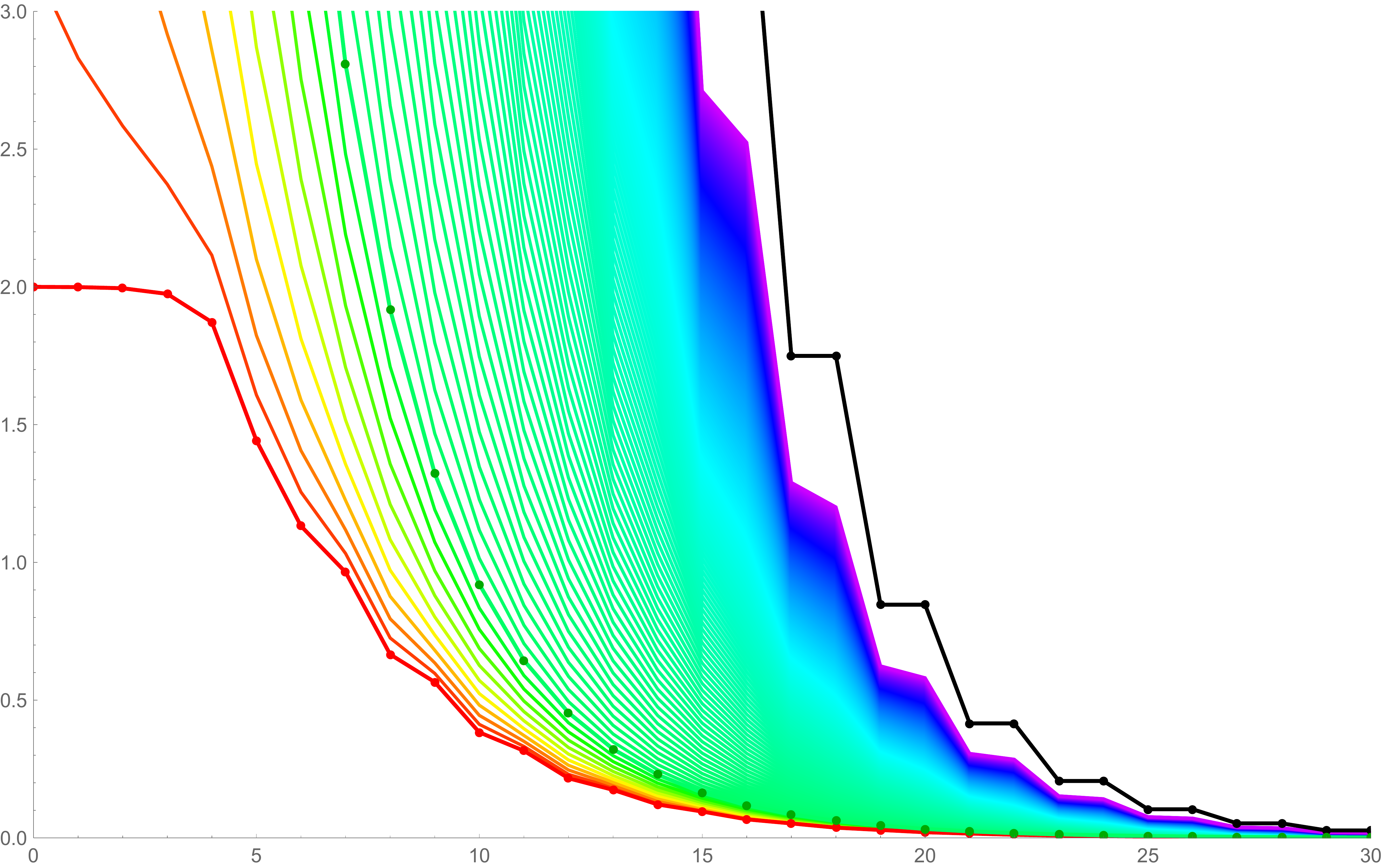}};
    \begin{scope}[shift={(plot1.south west)}]
    \node at (.8,5.) {\color{red}$L^1$};
    \node at (2.2,7.04) {\color[rgb]{0.8,1,0.}$L^{\!\frac32}$};
    \node at (2.85,7.) {\color[rgb]{0,1,0.4}$L^2$};
    \node at (3.7,7.) {\color[rgb]{0,1,0.504}$L^3$};
    \node at (4.53,7.) {\color[rgb]{0.,1.,.71}$L^{5}$};
    \node at (5.65,7) {\color[rgb]{0.8,0,1}$L^{\!25}$};
    \node at (6.2,7) {\color{black}$L^\infty$};
    \end{scope}
  \end{tikzpicture}
\end{center}
\vspace{-0.25cm}
\caption{The $L^p$-distance ($p\geq 1$) from equilibrium of \SRW\ on the LPS graph on $\mathrm{PSL}(2,\F_{29})$ shown in Fig.~\ref{fig:lps-ball4}, highlighting $p=1,2,\infty$.}
\vspace{-0.25cm}
\label{fig:lp-cutoff}
\end{figure}
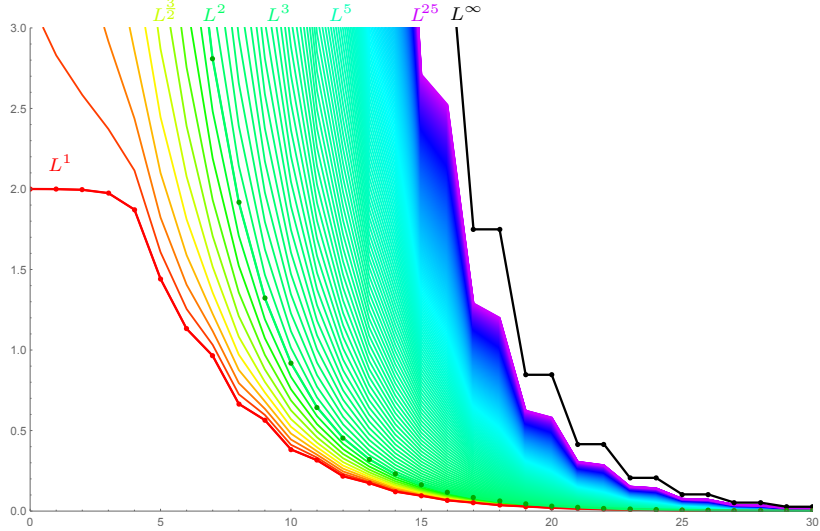

\begin{lemma}\label{lem:srw-l2}
  Fix $d\geq 3$ and let $\rho=2\sqrt{d-1}/d$. For every fixed $\epsilon>0$ and every connected $d$-regular graph $G$ on $n$ vertices, the \emph{\SRW} satisfies
  \[\tmix^{{\Ltwo}}(\epsilon) \geq  \tfrac12 \log_{1/\rho}n-O(\log\log n)\,.\]
  Moreover, if $G$ is non-bipartite Ramanujan then this is tight: \emph{\SRW} has an $L^2$-cutoff at $\frac12\log_{1/\rho}n > (1+\eta)\frac{d}{d-2}\log_{d-1}n $ for some constant $\eta=\eta(d)>0$.
\end{lemma}

\begin{proof}[\textbf{\emph{Proof of Lemma~\ref{lem:srw-l2}}}]
Let $P^t$ be the $t$-step transition kernel of \SRW, and let $\pi$ be the uniform distribution on $V(G)$.
For any $x\in V(G)$,
\[ \sum_y P^t(x,y)^2 = P^{2t}(x,x) \geq Q^{2t}(\xi,\xi)\,,\]
where $Q^t$ is the $t$-step transition kernel of \SRW\ on $\T_d$, the infinite $d$-regular tree rooted at $\xi$;
indeed, as argued above, if $\cX_t$ is \SRW\ on the cover tree $\T_d$ then $X_t=\phi(\cX_t)$ is \SRW\ on $G$, where $\phi$ is the cover map, and in particular a return to the root in the former implies a return to the origin in the latter.

The probability $Q^{2t}(\xi,\xi)$ is nothing but the probability of a 1\textsc{d} biased walk, reflected at 0, to be 0 at time $2t$, well-known (cf.~\cite[\S5, p128]{Woess09}) to be
\begin{equation}
   \label{eq-return-prob}
   Q^{2t}(\xi,\xi) = \frac{2\rho^2}{1-\rho} \frac{\rho^{2t}}{t\, 2^{2t-2}} \binom{2t-2}{t-1}\sim \frac{2\rho^2}{1-\rho^2} \frac{\rho^{2t}}{\sqrt{\pi t^3}}\qquad\mbox{ for }\rho = \frac{2\sqrt{d-1}}d\,.
 \end{equation}
In particular, using the standard expansion of the $L^2$-distance,
\begin{equation}
  \label{eq-l2-expansion}
  \sum_x \bigg(\frac{\mu(x)}{\pi(x)} - 1\bigg)^2 \pi(x) = \sum_x \frac{\mu^2(x)}{\pi(x)} - 1\end{equation}
which holds for every probability distribution $\mu$, thus we have
\begin{align*}
   \Big\| \frac{P^t(x,\cdot)}{\pi}-1\Big\|^2_{L^2(\pi)} \geq c_d\, n \rho^{2t} t^{-3/2}-1
   \end{align*}
for $c_d = 2\rho^2[(1-\rho^2)\sqrt{\pi}]^{-1}$. Consequently, from any initial $x\in V$ we have
 \[
 \tmix^{\Ltwo}(x,\epsilon) \geq \frac{\log (n/\epsilon) }{2\log(1/\rho)} - O(\log\log n)\,,
 \]
 where $\tmix^{\Ltwo}(x,\epsilon)$ is the first $t$ where $\|P^t(x,\cdot)/\pi-1\|_{L^2(\pi)}$ becomes at most $\epsilon$.

We next argue that
\begin{equation}
  \label{eq-d-infi}
  \tfrac12\log_{1/\rho}n > \tfrac{d}{d-2}\log_{d-1}n\qquad\mbox{ for every real $d\in(2,\infty)$ and $n\geq 2$}\,.
\end{equation}
Indeed,~\eqref{eq-d-infi} is equivalent to having $\frac{d-2}{d}\log(d-1) > 2\log\big(\frac{d}{2\sqrt{d-1}}\big)$ for all real $d\in(2,\infty)$, which, in turn, immediately follows from the fact that
\[ f(d):= \frac{d-2}d\log(d-1) - 2\log\Big(\frac{d}{2\sqrt{d-1}}\Big) \]
has $f'(d) = 2d^{-2}\log(d-1)$, so $f(2)=0$ whereas $f'(d) > 0$ for all $d>1$.

Finally, when $G$ is Ramanujan, the sought upper bound on the $L^2$-distance follows from considering the spectral representation (see, \emph{e.g.},~\cite{AF})
\begin{equation}
  \label{eq-spectral-Pt}
  \|P^t(x,\cdot)/\pi-1\|_{L^2(\pi)}^2 = n \sum_{i=2}^n |f_i(x)|^2 (\lambda_i/d)^{2t}
\end{equation}
for $\{f_i\}_{i=1}^n$ an orthonormal basis of eigenfunctions with eigenvalues $\{\lambda_i\}_{i=1}^n$ of the adjacency matrix and $\lambda_1=d$, and plugging in $|\lambda_i| \leq 2\sqrt{d-1}$.
\end{proof}

\begin{remark}
A different perspective on Lemma~\ref{lem:srw-l2} is given by the next proof of a slightly weaker statement.
  By the generalization by Serre~\cite{Serre97} (see~\cite[Theorem 1.4.9]{DSV03}) of the Alon--Boppana Theorem~\cite{Nilli91}, for every $\epsilon>0$ there exists $c_{\epsilon,d}>0$ such that $G$ has at least $c_{\epsilon,d} \, n$ eigenvalues $\lambda$ with $|\lambda|>2\sqrt{d-1}-\epsilon$.
Applying this fact for some $\epsilon(d)>0$ to be specified later, since
 $ \frac1n \sum_x \|\frac{P^t(x,\cdot)}{\pi}-1\|_{L^2(\pi)}^2 = \sum_{i=2}^n(\lambda_i/d)^{2t}$
 where $\lambda_2,\ldots,\lambda_n$ are the nontrivial eigenvalues of $G$ (this follows from~\eqref{eq-spectral-Pt} since an average over $x$ allows one to replace $\sum_x |f_i(x)|^2$ by $1$ for each $i$), we deduce that
 \[  \max_x \Big\|\frac{P^t(x,\cdot)}{\pi}-1\Big\|_{L^2(\pi)}^2 \geq \frac1n \sum_x \Big\|\frac{P^t(x,\cdot)}{\pi}-1\Big\|_{L^2(\pi)}^2 \geq c_{\epsilon,d}  n \Big(\frac{2\sqrt{d-1}-\epsilon}d\Big)^{2t}\!\!.\]
 Consequently,
 \begin{equation}\label{eq-t-lower-bnd}
 \tmix^{\Ltwo}(\delta) \geq \frac{\log (n/\delta)}{2\log\big(\frac{d}{2\sqrt{d-1}-\epsilon}\big)}-O(1)\,.
 \end{equation}
The proof now follows from~\eqref{eq-d-infi} as we may choose $\epsilon(d),\eta(d)>0$ so that the right-hand of~\eqref{eq-t-lower-bnd} would be at least $(1+\eta-o(1))\frac{d}{d-2}\log_{d-1} n$, as needed.
\end{remark}

For the general case of $p\in[1,\infty]$, we need the following simple claims.
\begin{claim}\label{clm:lower-bound}
  Let $G$ be a $d$-regular graph on $n$ vertices, and let $\T_d$ be the infinite $d$-regular tree rooted at $\xi$. For every $1 \leq p < \infty$, \emph{\SRW} on $G$ satisfies
  \[ \|P^t(x,\cdot)/\pi - 1\|_{L^p(\pi)}  \geq n^{(p-1)/p} \|Q^t(\xi,\cdot)\|_p-1\]
  for all $x,t$, where $P$ and $Q$ are the transition kernels of \emph{\SRW} on $G$ and $\T_d$.
\end{claim}
\begin{proof}
By the triangle inequality w.r.t.\ $\|\cdot\|_{L^p(\pi)}$,
\begin{align*}
  \|P^t(x,\cdot)/\pi -1\|_{L^p(\pi)} &\geq n^{(p-1)/p} \|P^t(x,\cdot)\|_{p}-1\,.
\end{align*}
Since $ P^t(x,y) = \sum_{\eta\in \phi^{-1}(y)} Q^t(\xi,\eta)$ for every cover map $\phi:V(\T_d) \to V(G)$ with $\phi(\xi)=x$, using the fact $(\sum_{i=1}^k a_i)^p \geq \sum_{i=1}^k a_i^p $ for every $a_1,\ldots,a_k>0$ and $p\geq 1$ gives
\[ \left(P^t(x,y)\right)^p \geq \sum_{\eta\in \phi^{-1}(y)} \left(Q^t(\xi,\eta)\right)^p\,.\]
Summing over all $y$ gives $\|P^t(x,\cdot)\|_p \geq \|Q^t(\xi,\cdot)\|_p$, as required.
\end{proof}

\begin{claim}
  \label{clm:Z-reflection}
Fix $d\geq 3$ and let $\T_d$ be the infinite $d$-regular tree rooted at $\xi$. There exist constants $c_1(d),c_2(d)>0$ such that, for all $k$ and $t$,
\[ c_1(d) \leq \frac{\P_\xi(|\cX_t|=k)}{\frac{k+1}t \P\left(Z_t = \frac{k+t}2\right)} \leq c_2(d) \]
where $|\cX_t|$ is the distance of $\cX_t$ from its origin $\xi$, and $Z_t\sim \Bin(t,\frac{d-1}d)$.
\end{claim}
\begin{proof}
The case $k=0$ follows from~\eqref{eq-return-prob}, since $\P(Z_{2t}=t) = (\frac{d-1}d)^t d^{-t} \binom{2t}t$, which is $(\rho/2)^{2t}\binom{2t}t$.
This extends to all $k$ using the decomposition
\[ \P_\xi(|\cX_t|=k) = \sum_{\ell=0}^{t-1} \P_\xi(|\cX_\ell|=0 ) \P_\xi\left(\{|\cX_j|>0\;:\;1\leq j \leq t-\ell\}\,, |\cX_{t-\ell}|=k\right)\]and the Ballot Theorem (see, \emph{e.g.},~\cite[\S{III.1}]{FellerI}).
\end{proof}
For more general local limit theorems on trees, see, \emph{e.g.},~\cite{Lalley93}.

\begin{figure}[t]
\vspace{-0.25cm}
\begin{center}
  \begin{tikzpicture}[font=\tiny]
    \node (plot1) at (0,0) {
    \includegraphics[width=0.6\textwidth]{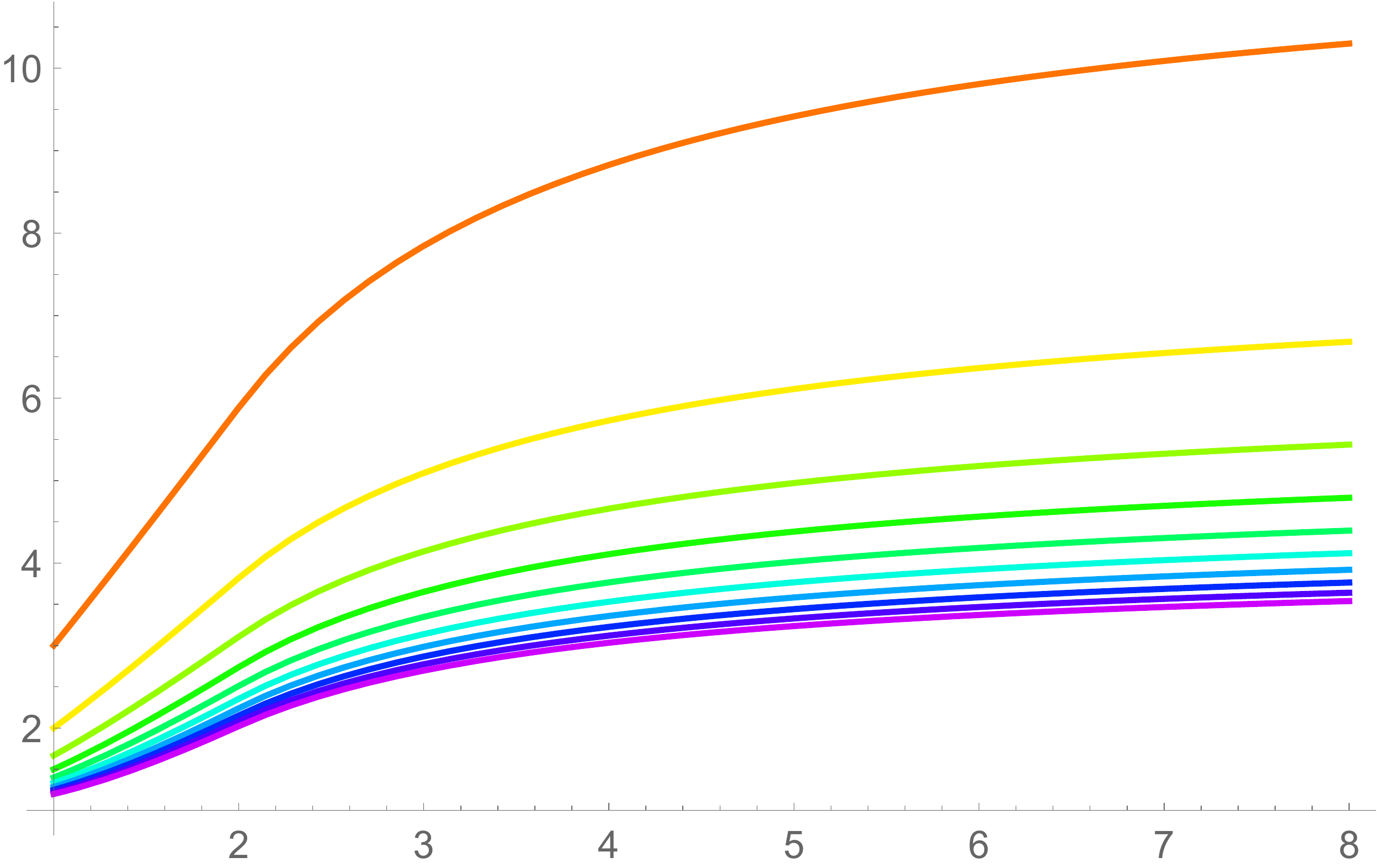}};
    \begin{scope}[shift={(plot1.south west)}]
    \node at (2.75,4.2) {\color[rgb]{1,.5,0.}$d=3$};
    \node at (5.65,1.25) {\color[rgb]{0.8,0,1}$d=12$};
    \end{scope}
  \end{tikzpicture}
\end{center}
\vspace{-0.45cm}
\caption{$L^p$-cutoff location (normalized by $\log_{d-1}n$) as a function of $p\geq 1$ for Ramanujan graphs with degree $d=3,\ldots,12$. The functions are $C^1$, but not $C^2$ at $p=2$.
}
\vspace{-0.25cm}
\label{fig:lpcutoffloc}
\end{figure}

\begin{proof}[\textbf{\emph{Proof of Proposition~\ref{prop:srw-lp}}}]
With Claims~\ref{clm:lower-bound} and~\ref{clm:Z-reflection} in mind, and using their notation, for every $t$ and $p\in[1,\infty]$ we have
\begin{align}
\|Q^t(\xi,\cdot)\|_p^p &= \sum_{k\geq 0} (d-1)^k \left( (d-1)^{-k} \P_\xi(|\cX_t|=k)\right)^p \nonumber\\
&\geq \Big(\frac{c_1(d)}t\Big)^p \sum_{k\geq 0} \left((d-1)^{k(1-p)/p} \,\P\left(Z_t=(k+t)/2\right)\right)^p\,.
\label{eq-Qt-lower-bound}
\end{align}
Writing $\beta t = (k+t)/2$ (so that $k=(2\beta-1)t$), the large deviation estimate
\begin{equation}
  \label{eq-Zt-ldp}
  \P(Z_t = \beta t) \asymp t^{-1/2}\exp[-H_e(\beta \;\|\;\tfrac{d-1}d)t]
\end{equation}
 for the binomial variable $Z_t$ thus leads to the following optimization problem:
\begin{equation}
  \label{eq-beta-min}
  \min\bigg\{ \frac{p-1}p(2\beta - 1) + H_{d-1}\left(\beta \;\|\; \tfrac{d-1}d\right) \;:\; \tfrac12\leq \beta \leq 1 \bigg\}\,.
\end{equation}
(Observe that in fact $\beta \leq \frac{d-1}{d}$ since for $\beta >\frac{d-1}d$ both terms are increasing.)
Let $f(\beta)$ denote the objective in~\eqref{eq-beta-min}. 
Then
\begin{align*}
  f'(\beta) &= \frac{2(p-1)}p + \log_{d-1}\Big(\frac{\beta}{(1-\beta)(d-1)}\Big)\,,
\end{align*}
and solving $f'(\beta)=0$ we get $\frac{1-\beta}\beta= (d-1)^{(p-2)/p}$. Since $f''(\beta)$ is positive, it follows that the minimizer of~\eqref{eq-beta-min} is at
\begin{equation}
  \label{eq-beta-minimizer}
  \beta^* = \frac{1}{(d-1)^{(p-2)/p}+1} \;\vee\;\frac12\,.
\end{equation}
(Observe that $\beta^*=1/2$ iff $p\geq 2$, hence the two regimes for the $L^p$-cutoff location as a function of $p$.)
By~\eqref{eq-Qt-lower-bound}--\eqref{eq-Zt-ldp}, for some $c=c_d>0$,
\[ \|Q^t(\xi,\cdot)\|_p \geq c_d\, t^{-3/2} (d-1)^{-f(\beta_*) t} \,,\]
and therefore, by Claim~\ref{clm:lower-bound}, for every starting vertex $x$,
\begin{equation}
  \label{eq-Pt-beta*-bnd}
  \Big\|\frac{P^t(x,\cdot)}{\pi}-1\Big\|_{L^p(\pi)} \geq c_d\, n^{(p-1)/p} t^{-3/2} (d-1)^{-f(\beta_*) t} -1 \,.
\end{equation}
This implies~\eqref{eq-tmix-lp-lower-[2,inf]} (and is furthermore valid for every starting vertex $x$).

For matching upper bounds in case $G$ is a Ramanujan graph, first take $ p \geq 2$.
The lower bound established above is $ D_p(t) \geq c_d\, n^{(p-1)/p} t^{-3/2}\rho^t - 1$.
Recalling Lemma~\ref{lem:srw-l2}, for Ramanujan graphs,
\[ D_2(t) \leq \sqrt{n} \rho^t\,,\qquad D_\infty(t) \leq D_2(\lfloor t/2\rfloor )^2 \leq n \rho^{t }\,,\]
using the well-known fact (a routine application of Cauchy--Schwarz) that
\begin{equation}
  \label{eq-dinf-d2}
  D_{\infty}(s+t) \leq D_2(t) D^*_2(s)\,,
\end{equation}
where $D^*_2(s)$ corresponds to the reversed chain (here $D_2(s)=D_2^*(s)$ as \SRW\ is reversible).
So, by the Riesz--Thorin Interpolation Theorem (see, \emph{e.g.}, \cite[Theorem~1.3, p.~179]{SteinWeiss71}), 
for $2\leq p \leq \infty$, we deduce that $D_p(t) \leq n^{(p-1)/p} \rho^{t}$.

Having established~\eqref{eq-tmix-lp-rama-[2,inf]} for $p\geq 2$, now take $1< p \leq 2$. Let
\begin{equation}
  \label{eq-P6-muk-def}
  \mbox{$ P^t(x,\cdot) = \sum_k \P_\xi(|\cX_t|=k) \tilde\mu_k(x,\cdot)$}
\end{equation}
where $\tilde\mu_k$ is the law of the projection of \NBRW\ on the endpoint of its directed edge, started at a uniform edge originating from $x$.
By Jensen's inequality,
\[  \big(d(d-1)^{k-1}\big)^{-1/p}\|\tilde\mu_k(x,\cdot)-\tfrac1n\|_p \leq
 \big(d(d-1)^{k-1}\big)^{-1/2}\|\tilde\mu_k(x,\cdot)-\tfrac1n\|_2\,.\]
 Notice that
 \[  n\Big\|\tilde\mu_k(x,\cdot)-\tfrac1n\Big\|_2^2 = \Big\|\frac{\tilde\mu_k(x,\cdot)}\pi - 1\Big\|_{L^2(\pi)}^2 \leq \max_{y : xy\in E(G)}\Big\|\frac{\mu_{k-1}\big((x,y),\cdot\big)}{\pi_{\vE}}-1\Big\|^2_{L^2(\pi_\vE)}\]
where $\mu_k$ is the $k$-step transition kernel of the \NBRW\ and $\pi_\vE$ is its stationary distribution.
In our analysis of the \NBRW\ in \S\ref{sec:nbrw}, we will show (see~\eqref{eq-d2(t)-bound}) that the right-hand side of the last display is
$O(n k^2 (d-1)^{-k})$, whence
\begin{equation*}
   \|\tilde\mu_k(x,\cdot)-\tfrac1n\|_p \leq c_d\, k (d-1)^{k (1-p)/p}\,.
\end{equation*}
Recalling~\eqref{eq-P6-muk-def}, it now follows that
\begin{align*}
  \|P^t(x,\cdot)-\tfrac1n\|_p &\leq (t+1)\max_{0\leq k \leq t} \left(\P_\xi(|\cX_t|=k)\right) \|\tilde\mu_k(x,\cdot)-\tfrac1n\|_p  \\
  &\leq c'_d t^{2} \max_{0\leq k \leq t} (d-1)^{k(1-p)/p} \left(\P_\xi(|\cX_t|=k)\right) \,,
\end{align*}
which, in view of~\eqref{eq-Qt-lower-bound}, gives rise to the same optimization problem~\eqref{eq-beta-min}. Therefore, the right-hand side of the last display is at most $t^{C} (d-1)^{-f(\beta_*) t}$ for $C>0$ fixed. Taking $t$ as in~\eqref{eq-tmix-lp-rama-[2,inf]} with a suitable additive $O(\log\log n)$ term gives $ (d-1)^{f(\beta_*) t} \geq n^{(p-1)/p} t^{2C}$. Thus,
\vspace{-0.03in}
\[
  \|P^t(x,\cdot)/\pi -1\|_{L^p(\pi)} =  n^{(p-1)/p} \|P^t(x,\cdot)-\tfrac1n \|_p \leq t^{-C}\,,
\vspace{-0.03in}
\]
establishing~\eqref{eq-tmix-lp-rama-[2,inf]} for all $1< p \leq 2$.
\end{proof}

\section{Nonbacktracking walks}\label{sec:nbrw}
\subsection{Spectral decomposition}\label{subsec:decomposition}
The spectrum of the nonbacktracking walk has been thoroughly studied, in part due to the fact that its eigenvalues are precisely the inverse of the poles of the so-called Ihara Zeta function of the graph (cf.~\cite{Bass92,KS00}). Our analysis here, on the other hand, hinges on the structure of the eigenfunctions, starting with a spectral decomposition of the nonbacktracking operator; this builds on properties of this operator that appear implicitly in~\cite{KS00} (see also~\cite{ABLS07,AFH15,Bass92} as well as~\cite[Exercise~6.59]{LP}).
Proposition~\ref{prop:spectral} below gives a more complete picture.

Throughout this section, for a graph $G=(V,E)$, we denote its adjacency matrix by $A=A(G)$ and let $\lambda_1=d \geq \lambda_2 \geq \ldots\geq \lambda_n$ be its eigenvalues. Denote by $\vE$ the set of $N=2|E|$ directed edges of $G$; we refer to undirected edges as $xy\in E$ and to directed ones as $(x,y)\in \vE$ for the sake of clarity. The nonbacktracking walk matrix $B$ is the $(\vE\times\vE)$-matrix given by
\begin{equation}\label{eq-B-def} B_{(u,v),(x,y)}=\one_{\{v=x\,,\,u\neq y\}}\qquad \mbox{ for $(u,v)$ and $(x,y)$ in $\vE$}\,.\end{equation}
Though $B$ may not be a normal operator, it can be decomposed as follows.
\begin{proposition}\label{prop:spectral}
Let $G=(V,E)$ be a connected $d$-regular graph ($d\geq 3$) on $n$ vertices. Let $N=dn$ and let $\{\lambda_i\}_{i=1}^n$ be the eigenvalues of the adjacency matrix, with $\lambda_1=d$. Then the operator $B$ from~\eqref{eq-B-def} is unitarily similar to
\begin{equation}
  \label{eq-B-decomposition}
   \Lambda=\operatorname{diag}\Bigg( d-1,\begin{pmatrix}
\theta_2 & \alpha_2  \\
0 & \theta'_2
\end{pmatrix},\ldots, \begin{pmatrix}
\theta_n & \alpha_n
 \\ 0 & \theta'_n
\end{pmatrix},\overbrace{\Bigg.-1,\ldots,-1}^{N/2-n},\overbrace{\Bigg.1,\ldots,1}^{N/2-n+1}\Bigg)
\end{equation}
where $|\alpha_i|<2(d-1)$ for all $i$ and $\theta_i,\theta_i'\in \C$ are defined as the solutions to
\begin{equation}
  \label{eq-theta-lambda-quadratic}
  \theta^2 - \lambda_i \theta + d-1 = 0\,.
\end{equation}
\end{proposition}
\begin{remark}
  \label{rem-precise-alpha}
  The exact value of $|\alpha_i|$ is shown in Proposition~\ref{prop-alphai}
  to be $d-2$ for every $|\lambda_i|\leq2\sqrt{d-1}$ and $\sqrt{d^2-\lambda_i^2}$ for every $2\sqrt{d-1}<|\lambda_i|<d$.
\end{remark}
\begin{remark}
  \label{rem-theta-norm}
We see that every eigenvalue $\theta \neq \pm 1$ of $B$ is of the form $\lambda/2 \pm \sqrt{(\lambda/2)^2-(d-1)}$ for some eigenvalue $\lambda$ of $A$
(with $\theta=d-1$ matching the principal eigenvalue $\lambda=d$).
Indeed, this well-known fact follows from Bass's Formula~\cite{Bass92}, which in the $d$-regular case is equivalent to the statement that
$ f_B(\theta) = \left(1-\theta^2\right)^{N/2-n} f_A(\theta+(d-1)/\theta)$
for $f_A$ and $f_B$ the characteristic polynomials of $A$ and $B$, respectively.
\begin{compactenum}[(i)]
\item $\lambda= d$ corresponds to $\theta=d-1$, the principal eigenvalue of $B$ matching the eigenvector $w_1 \equiv N^{-1/2}$; the second solution, $\theta'=1$, was already accounted for in~\eqref{eq-B-decomposition}. An eigenvalue of $\lambda=-d$ (when $G$ is bipartite) yields $\theta=-(d-1)$ and an extra $-1$ eigenvalue ($N-n+1$ altogether).
  \item $2\sqrt{d-1}<|\lambda|<d$ yields two eigenvalues $\theta\neq \theta'\in \R$ of $B$.
  \item $\lambda < |2\sqrt{d-1}|$ yields $\theta=\bar\theta'\in\C\setminus\R $ with  $|\theta|=\sqrt{d-1}$ (for instance, $\lambda=0$ corresponds to $\theta=i\sqrt{d-1}$ and $\theta'=-i\sqrt{d-1}$).
  \item $\lambda=\pm2\sqrt{d-1}$ gives a single solution
$\theta=\pm\sqrt{d-1}$ with multiplicity 2.
\end{compactenum}
\end{remark}
\begin{remark}\label{rem:eigenvec}
For each $\theta\in\C$, define $T_\theta:\ell^2(V)\to\ell^2(\vE)$  by
\begin{equation}
  \label{eq-T-def}
  (T_\theta f)(x,y) := \theta f(y) - f(x)\,.
\end{equation}
Each solution $\theta\neq\pm1$ of equation~\eqref{eq-theta-lambda-quadratic}, for some $\lambda$ such that $A f = \lambda f$,
is an eigenvalue  of $B$ corresponding to the eigenvector $T_{\theta} f$; indeed,
\begin{align*}
  (B T_{\theta} f)(x,y) &= \sum_{\substack{z:yz\in E\\ z\neq x}} \!\left(\theta f(z)-f(y)\right) = \theta [(Af)(y)-f(x)] - (d-1)f(y)\nonumber\\
  &= \left[\theta\lambda - (d-1)\right] f(y) - \theta f(x) = \theta (T_{\theta} f)(x,y)\,;
\end{align*}
where the last equality used~\eqref{eq-theta-lambda-quadratic} to replace $\theta\lambda$ by $\theta^2+d-1$;
thus, $T_{\theta} f$ is an eigenfunction of $B$ corresponding to $\theta$ as long as $T_{\theta} f\neq 0$, and clearly $T_{\theta} f = 0$ only if $\theta = \pm 1$ (which, in turn, occurs iff $\lambda=\pm d$).
\end{remark}
\begin{proof}[\textbf{\emph{Proof of Proposition~\ref{prop:spectral}}}]
Observe that $\ell^2(\vE) = \ell^2_+(\vE) \oplus \ell^2_-(\vE)$ where
\[ \ell^2_+(\vE) = \left\{ w \;:\; w(x,y) = w(y,x)\right\} \,,\quad
\ell^2_-(\vE)= \left\{ w \;:\; w(x,y) = -w(y,x)\right\}\,, \]
as the term for $(x,y)$ in $\left<w_+,w_-\right>$ cancels with that of $(y,x)$
if $w_\pm \in\ell^2_\pm(\vE)$.

With this in mind, the eigenspaces of $1$ and $-1$ in $B$ are straightforward:
the \emph{star spaces} $\cS_- \subset \ell^2_-(\vE)$ and $\cS_+ \subset \ell^2_+(\vE)$ are defined by
\[ \cS_\pm = \Span\left(\left\{ s^\pm_x : x\in V\right\}\right)\,,\quad\mbox{ where }\quad s^\pm_x(u,v)=\left\{\begin{array}{ll}
  1 & u=x\,, \\
  \pm1 & v = x\,,\\
  0 & \mbox{otherwise}\,.
\end{array}\right.
\]
For every $w\in \ell^2_-(\vE)$ and $s^-_x$ as above $\left<w,s^-_x\right> =2\sum_{y:xy\in E} w(x,y)$, and so
$(B w)(x,y) = -w(y,x) = w(x,y)$ when in addition $w \perp s^-_y$. Thus,
\begin{equation}
  \label{eq-B-eig1}
  Bw = w\qquad\mbox{ for every }\quad w\in \ell^2_-(\vE) \cap \cS_-^\perp\,,
\end{equation}
and similarly,
\begin{equation}
  \label{eq-B-eig-1}
 Bw = -w\qquad\mbox{ for every }\quad w\in \ell^2_+(\vE) \cap \cS_+^\perp\,.
\end{equation}
As for the dimension of these spaces, note that if $\{a_x\}_{x\in V}$ is such that $\sum a_x s^-_x =0$ then $a_x = a_y$ for every $xy\in E$; since $G$ is connected, this implies that $\dim(\cS_-) = n-1$, thus $B$ has an orthonormal system of $N/2 - (n-1)$ eigenvectors with eigenvalue $1$. Similarly, if $\sum a_x s^+_x = 0$ then $a_x = -a_y$ for every $xy\in E$, so the eigenspace of $-1$ has dimension $N/2-(n-1)$ if $G$ is bipartite and dimension $N/2-n$ otherwise.

Having specified these eigenvectors of $B$ as well as those corresponding to $\theta_i,\theta'_i$ in Remark~\ref{rem:eigenvec}, we proceed to analyzing their inner products.
Observe that after appropriate permutations of its rows and columns, $B$ becomes block diagonal with blocks $J_d - I_d$, where $J_d$ and $I_d$ are the all-one matrix and identity matrix of order $d$, respectively; thus, $B$ has an inverse, which under the same permutations is block diagonal with blocks $(d-1)^{-1} J_d - I_d$, so the matrix
$ C := (d-1) B^{-1} + B $ (which, of course, satisfies $C w = ((d-1)/\theta + \theta)w$ for every eigenfunction $w$ of $B$ with eigenvalue $\theta$)
is given by
\[ C_{(u,v),(x,y)} = \left\{\begin{array}{ll}
  1 & \mbox{$\{v=x,\,u\neq y\}$ or $\{y=u,\,v\neq x\}$}\,,\\
  -(d-2) & (u,v)=(y,x)\,,\\
  0 & \mbox{otherwise}\,.
\end{array}\right.
\]
Thus, $C$ is real symmetric, and $\ell^2_+(\vE)$ and $\ell^2_-(\vE)$ are invariant under it.
Furthermore, if $f\in\ell^2(V)$ and $w_f\in\ell^2(\vE)$ is given by $w_f(x,y):=f(y)$ then
\begin{align*}
  (C w_f)(x,y) &= \sum_{\substack{z : yz\in E\\ z\neq x}}f(z) + \sum_{\substack{v:vx\in E \\ v\neq y}} f(x) - (d-2)f(x) \\
&= \sum_{\substack{z:yz\in E\\ z\neq x}} f(z) + f(x) = (Af)(y)\,,
\end{align*}
and similarly, if $w'_f(x,y) := f(x)$ then $(C w'_f)(x,y) = (Af)(x)$.
Moreover,
$ \left<w_f,w_g\right> = \big <w'_f,w'_g\big> = d \left<f,g\right>$  and $\left<w_f,w'_g\right> = \left<f,A g\right>$
for $f, g \in \ell^2(V)$.

In particular, the eigenfunctions $( f_i )_{i=1}^n$ correspond in this way to pairwise orthogonal eigenspaces of $C$ with eigenvalues $(\lambda_i)_{i=1}^n$; the dimension of each eigenspace is $1$ if $\lambda_i=\pm d$ and $2$ otherwise (as before, $w_{f}$ can be a multiple of $ w'_{f}$ only if $w\equiv c$ or when $G$ is bipartite and $w\equiv c$ on one part and $w\equiv -c$ on the other), and they notably include the eigenfunctions $T_{\theta_i} f_i$ of $B$.

Of course, every such 2-dimensional eigenspace corresponding to $\lambda_i \neq \pm d$ is orthogonal to the eigenvectors of $B$ from~\eqref{eq-B-eig1}--\eqref{eq-B-eig-1} (corresponding to the eigenvalues $\pm1$), as those are also eigenvectors of $\pm d$ for the self-adjoint $C$.
Finally, the eigenvector $w \equiv 1$ with the eigenvalue $d-1$ of $B$ (and eigenvalue $d$ of $C$) is orthogonal to $\ell^2_-(\vE)$ (thus to all eigenvectors from~\eqref{eq-B-eig1}), whereas if $G$ is bipartite and we take $w \equiv 1$ on outgoing edges from a prescribed part of $G$ and $w \equiv -1$ on the incoming ones (with eigenvalue $-(d-1)$ of $B$) then $w \perp \ell^2_+(\vE)$, thus it is orthogonal to all eigenvectors from~\eqref{eq-B-eig-1}.

Suppose for now that $A$ has no eigenvalue $\lambda_i$ such that $|\lambda_i|=2\sqrt{d-1}$. Then there are two distinct solutions to~\eqref{eq-theta-lambda-quadratic} for each of the $\lambda_i$'s, and so, in particular, the eigenspace of $C$ corresponding to $\lambda_i\neq \pm d$ has two linearly independent eigenvectors of $B$---corresponding to eigenvalues $\theta_i$ and $\theta_i'$. The orthogonality of the eigenspaces from the discussion above now establishes the form of $\Lambda$ from~\eqref{eq-B-decomposition}.

 When there exist eigenvalues of $A$ such that $|\lambda_i|=2\sqrt{d-1}$, we have the unique solution $\theta_i=\lambda_i/2$ for~\eqref{eq-theta-lambda-quadratic}, and claim that this gives rise to a Jordan block $\left(\begin{smallmatrix}
  \lambda_i/2 &1 \\0&\lambda_i/2
\end{smallmatrix}\right)$.
Indeed, recalling that $B T_\theta f_i = \theta_i f_i$, observe that
\begin{align}
(B T_{1+\theta_i}f_i)&(x,y) = [(1+\theta_i) \lambda_i - (d-1)] f_i(y) - (1+\theta_i)f_i(x) \nonumber\\
&= \theta_i \big[ (1+\theta_i)f_i(y)-f_i(x)\big] + \big[\theta_i^2 +\theta_i-(d-1)\big]f(y)-f(x) \nonumber\\
&= \theta_i (T_{1+\theta_i}f_i)(x,y) + (T_{\theta_i}f_i)(x,y)\,,\label{eq-jordan}
\end{align}
where the second equality used $\theta_i = \lambda_i / 2$ and the last one used $\theta_i^2 = d-1$.
As these both belong to the corresponding eigenspace of $C$, we arrive at~\eqref{eq-B-decomposition}.

To conclude the proof, it remains to show that $|\alpha_i|< 2(d-1)$ if $\lambda_i \neq \pm d$. Recall that there exist unit vectors $w_i,w'_i$ such that $ B w'_i = \alpha_i w_i + \theta'_i w'_i$
(these can be taken as columns $2i$ and $2i+1$ of $U$ as above).
Hence,
\begin{equation}
  \label{eq-B-w-w'-relation}
  (B - \theta'_i I) w'_i = \alpha w_i\,.
\end{equation}
Let $\|\cdot\|_{2\to2}$ be the $\ell^2(\vE)\to\ell^2(\vE)$ operator norm; we claim $\|B\|_{2\to2}=d-1$.
 Indeed, it is easy to verify that
\[ (B B^*)_{(u,v),(x,y)} = \left\{\begin{array}{ll}
  d-1 & \mbox{$x=u$ and $y=v$}\,,\\
  d-2 & \mbox{$x\neq u$ and $y=v$}\,,\\
  0 & \mbox{otherwise}\,.
\end{array}\right.
\]
We see that $B B^*$ has $\|B B^*\|_{\infty\to\infty} = (d-1)^2$ and an eigenvalue $(d-1)^2$ corresponding to the eigenvector $w\equiv 1$; thus, $\|B\|_{2\to2} =d-1$.
By~\eqref{eq-B-w-w'-relation}, using $|\theta'_i|< d-1$ and $\|w_i\|=\|w'_i\|=1$, and we infer that $|\alpha|< 2(d-1)$,
concluding the proof of the proposition.
\end{proof}

\subsection{Cutoff on non-bipartite Ramanujan graphs}\label{subsec:nonbipartite-nbrw}
On every $d$-regular graph on $n$ vertices, the number of directed edges at distance $\ell$ from a given $(x,y)\in\vE$ is at most $(d-1)^\ell$; this readily implies (as stated in~\cite[Claim~4.8]{LS-gnd}) that the nonbacktracking random walk satisfies
\begin{equation}
  \label{eq-nbrw-lower-bound}
  \tmix(1-\epsilon) \geq \lceil \log_{d-1}(dn) \rceil -\lceil \log_{d-1}(1/\epsilon)\rceil \quad\mbox{ for any $0<\epsilon<1$}\,.
\end{equation}
Our goal in this section is to show an asymptotically tight upper bound on $\tmix$ using the spectral decomposition of the nonbacktracking operator $B$.
\begin{theorem}
  \label{thm-nbrw}
 Let $G$ be a non-bipartite Ramanujan graph on $n$ vertices with degree $d\geq 3$. Let
 $\mu_{t}$ be the $t$-step transition kernel of the \emph{\NBRW}, and let $\pi$ be the uniform distribution on $\vE$.  Then for some fixed $c(d)>0$,
 \[ \max_{(x,y)\in\vE} \Big\| \frac{\mu_t\big((x,y),\cdot\big)}{\pi} - 1 \Big\|^2_{L^2(\pi)} \leq \frac{c(d)}{\log n}~ \mbox{ at }~ t = \left\lceil \log_{d-1} n + 3\log_{d-1}\log n\right\rceil\,. \]
  Consequently, on any sequence of such graphs, the \emph{\NBRW} exhibits $L^1$-cutoff and $L^2$-cutoff both at time $\log_{d-1} n$.
\end{theorem}
\begin{remark}
  \label{rem-nbrw-constant}
  The constant $c(d)$ in the above theorem can be taken to be $ 8(d-1)\log^{-2}(d-1) + 1$ for any sufficiently large enough $n$ (cf.~\eqref{eq-d2(t)-bound} below).
\end{remark}
\begin{proof}[\emph{\textbf{Proof of Theorem~\ref{thm-nbrw}}}]
Appealing to Proposition~\ref{prop:spectral}, let $U$ be the unitary matrix such that $B = U \Lambda U^*$ with $\Lambda$ from~\eqref{eq-B-decomposition}, and write
\[ U = \left( w_1 \mid w_2\mid w'_2 \mid w_3 \mid w'_3\mid \ldots \mid w_n \mid w'_n \mid u_1 \mid \ldots u_{N-(2n-1)} \right)\,,\]
in which $w_1 \equiv N^{-1/2} $. Recalling Remark~\ref{rem-theta-norm}, observe that the assumption that $G$ is non-bipartite Ramanujan implies that for all $i=2,\ldots,n$, the solutions $\theta_i,\theta_i'$ to~\eqref{eq-theta-lambda-quadratic} satisfy $\theta_i' = {\bar\theta_i}$ and $|\theta_i| = \sqrt{d-1}$.

Let $(x_0,y_0)\in \vE$ be some initial edge for the \NBRW; by the expansion~\eqref{eq-l2-expansion} of the $L^2$-distance, the $t$-step transition kernel $\mu_t = (d-1)^{-t} B^t$ satisfies
\begin{align}
   \Big\| \frac{\mu_t\big((x_0,y_0),\cdot\big)}{\pi}-1\Big\|^2_{L^2(\pi)}  &= N \sum_{(x,y)}\left|\mu_t\big((x_0,y_0),(x,y)\big)\right|^2  - 1\nonumber\\
   &= N  \left\| \mu_t\big((x_0,y_0),\cdot\big)\right\|^2 - 1\,.
 \label{eq-l2-dist}
   \end{align}
Using $B=U\Lambda U^*$ with $\Lambda$ from~\eqref{eq-B-decomposition} and $U$ as specified above we find that
\begin{align*}
  B^t\big((x_0,y_0),\cdot\big)  &= (d-1)^t w_1(x_0,y_0) w_1 + \sum_{i} (\pm1)^t u_i(x_0,y_0)  u_i \\
  &+\sum_{i=2}^{n} \theta_i^t w_i(x_0,y_0) {\bar w_i} + \left(\bar\theta_i^t w'_i(x_0,y_0) + \gamma_{i}(t) w_i(x_0,y_0)\right){\bar w_i'}\,,
\end{align*}
where
\[
  \gamma_{i}(t) := \alpha_i \sum_{j=0}^{t-1} \theta_i^{j}{\bar\theta_i}^{t-1-j}
\]
with $\alpha_i$ from Proposition~\ref{prop:spectral}. Note that in particular, as $\alpha_i<2(d-1)$,
\begin{equation}
  \label{eq-gamma-bound}
  |\gamma_i(t)| \leq 2(d-1) t |\theta_i|^{t-1}\,.
\end{equation}
From the above expansion of $B^t$, since $U$ is unitary and $w_1 \equiv N^{-1/2}$,
\begin{align}
&\left\|\mu_t\big((x_0,y_0),\cdot\big)\right\|^2 = \frac1N + \sum_{i}(d-1)^{-2t} |u_i(x_0,y_0)|^2 \nonumber\\
&~\quad+ (d-1)^{-2t} \sum_{i=2}^n \left( |\theta_i|^{2t} |w_i(x_0,y_0)|^2 + \left|{\bar \theta_i^t}w'_i(x_0,y_0) + \gamma_{i}(t)w_i(x_0,y_0)\right|^2\right)\,.
\label{eq-mut-spectral1}
\end{align}
Now we exploit the fact that $G$ is Ramanujan: since $|\theta_i |=\sqrt{d-1}$ for every $2\leq i\leq n$, the expression in the second line of~\eqref{eq-mut-spectral1} is at most
\[(d-1)^{-t} \sum_{i=2}^n  |w_i(x_0,y_0)|^2 + 2|w'_i(x_0,y_0)|^2 + 2\frac{|\gamma_{i}(t)|^2}{(d-1)^t} |w_i(x_0,y_0)|^2\,,\]
using the parallelogram law.
Since by Parseval's identity,
\[\sum_i |u_i(x_0,y_0)|^2 + \sum_i |w_i(x_0,y_0)|^2 + |w'_i(x_0,y_0)|^2 = \|\delta_{(x_0,y_0)}\|^2 = 1\,,\]
and with~\eqref{eq-l2-dist} in mind, we infer that
\begin{equation*}
\Big\| \frac{\mu_t\big((x_0,y_0),\cdot\big)}{\pi}-1\Big\|^2_{L^2(\pi)}  \leq 2N (d-1)^{-t} \left(1 + \max_i \frac{|\gamma_{i}(t)|^2}{ (d-1)^t}\right)\,.
\end{equation*}
Substituting the bound~\eqref{eq-gamma-bound} on $\gamma_i(t)$, again using that $G$ is Ramanujan,
\begin{equation}
  \label{eq-d2(t)-bound}
  \Big\| \frac{\mu_t\big((x_0,y_0),\cdot\big)}{\pi}-1\Big\|^2_{L^2(\pi)}   \leq  2N (d-1)^{-t} \left(4(d-1) t^2 + 1\right)\,.
\end{equation}
In particular, for $t=\lceil \log_{d-1}n + 3\log_{d-1}\log n\rceil$,
\[
   \Big\| \frac{\mu_t\big((x_0,y_0),\cdot\big)}{\pi}-1\Big\|^2_{L^2(\pi)}  \leq O(1/\log n) = o(1)\,,
\]
thus concluding the proof of Theorem~\ref{thm-nbrw}.
\end{proof}

Using the reduction in~\S\ref{sec:reduction} from \SRW\ to \NBRW\ (see~\eqref{eq-dtv-upper}--\eqref{eq-dtv-lower}), one can deduce Theorem~\ref{thm-srw} from Theorem~\ref{thm-nbrw}, as the $O(\log\log n)$ window for the \NBRW\ is negligible compared with the term $s\sqrt{\log_{d-1} n}$ in~\eqref{eq-dtv-srw}.

Note that for every integer $\ell\geq (2+o(1))\log_{d-1} n$ there is a path of length exactly $\ell$ between every pair of vertices $x,y$
using $D_\infty(2t) \leq D_2(t) D_2^*(t) $ for the \NBRW\ (recall~\eqref{eq-dinf-d2}, and that the chain and its reversal are isomorphic).

\begin{proof}[\textbf{\emph{Proof of Corollary~\ref{cor-dist}}}] Since $\max_{(x,y)}\|\mu_t((x,y),\cdot) - \pi\|_\tv = o(1)$
at time $t$ as per Theorem~\ref{thm-nbrw}, for every $x$, all but $o(n)$ directed edges can be reached by a nonbacktracking path of length $t$ from $x$. The remark above~\eqref{eq-nbrw-lower-bound} on the growth of balls in a $d$-regular graph thus implies the corollary: the statement on a nonbacktracking cycle follows from applying this argument once on a directed edge originating from $x$ (and reaching almost every $y$ within the proper length bound) and once on an arbitrarily chosen other directed edge ending at $x$, in the reversed \NBRW.
\end{proof}

\begin{proof}[\textbf{\emph{Proof of Corollary~\ref{cor-disjoint}}}]
Note that at time $R$, the $L^2$-distance of the \NBRW\ from equilibrium is $O(1/\log^{3/2}n)$ by~\eqref{eq-d2(t)-bound}, and that
$k =O(\log n)$ since $k\leq g$. For a uniformly chosen path $(y_i)_{i=1}^k$ in $G$, each $y_i$ is uniform by the stationarity of the \NBRW. Thus, by a union bound over the vertices $y_i$, for each $i$ there exists a path of length $R$ from the edge $(x_i,z_i)$ to $(y_i,z'_i)$, except with probability $O(k /\log^{3/2} n) = o(1)$, where $z_i$ and $z'_i$ are not on the paths $(x_i)$ and $(y_i)$, respectively. The conclusion now follows since, if vertex $\ell$ of the path from $x_i$ coincides with vertex $\ell'$ of the path from $x_{j}$, then $\ell+\ell'+k > g$ and $(R-\ell)+(R-\ell')+k> g$, so $k> g-R$, a contradiction.
\end{proof}

\begin{remark}\label{rem:l2-transitive}
In the setting of Theorem~\ref{thm-nbrw}, if $G$ is in addition transitive then, by using the exact value $|\alpha_i|=d-2$ from Proposition~\ref{prop-alphai} below, the $L^2$-mixing time of the \NBRW\ can be pinpointed precisely: let
\[ \mbox{$\Upsilon_G(k) := (d-2)^2 (d-1)^{-1}\int U_{k-1}(x)^2 d\mu_G $} \]
for $\mu_G = \frac1n \sum_{i} \delta_{\lambda_i/(2\sqrt{d-1})}$ the empirical spectral distribution (ESD) of $G$ and $U_k(\cos(x))=\frac{\sin((k-1)x)}{\sin x}$ the Chebyshev polynomial of the second kind. Then for any fixed $\epsilon>0$,
    \begin{equation}
      \label{eq-NBRW-L2-dist-precise}
      \tmix^{\Ltwo}(\epsilon) = \left\lceil \log_{d-1}(n) + \log_{d-1}\left(\Upsilon_G(\log_{d-1}n)+2\right) + \log_{d-1}(1/\epsilon)\right\rceil\,.
    \end{equation}
Indeed, from~\eqref{eq-mut-spectral1} we see that for any non-bipartite Ramanujan graph $G$ (not necessarily transitive), averaging over the initial state $(x_0,y_0)$ gives
\begin{align*}
   \frac{1}N\sum_{(x_0,y_0)} \left\|\mu_t\big((x_0,y_0),\cdot\big)\right\|^2 &= \frac1N + \frac{N-2n+1}{(d-1)^{2t}} + \frac{2n-2}{(d-1)^{t}} + \frac{\sum_i |\gamma_i(t)|^2}{(d-1)^{t}}\,,
 \end{align*}
using that $w_i\perp w'_i$ and $\|u_i\|=\|w_i\|=\|w'_i\|=1$ for all $i$. Thus, by~\eqref{eq-l2-dist},
\begin{align*}
  \frac1N \sum_{(x_0,y_0)} \Big\| \frac{\mu_t\big((x_0,y_0),\cdot\big)}{\pi}-1\Big\|^2_{L^2(\pi)}
  &= (1+o(1))\bigg(\sum_i |\gamma_i(t)|^2 +2 \bigg) \frac{n}{(d-1)^{t}}\,,
\end{align*}
provided that $t\to\infty$ with $n$. Writing $\varphi_i=\lambda_i/(2\sqrt{d-1})$ (so $\theta_i = \cos\varphi_i$ for $i=2,\ldots,n$) and using Proposition~\ref{prop-alphai},
\[ |\gamma_i(t)| = (d-2) \bigg|\frac{\bar\theta_i^t - \theta_i^t}{\bar\theta_i-\theta_i}\bigg| = (d-2)(d-1)^{(t-1)/2} \bigg|\frac{\sin(t\varphi_i)}{\sin \varphi_i}\bigg|\,,\]
which implies the analogue of~\eqref{eq-NBRW-L2-dist-precise} for the average of the mixing times over the initial states $(x_0,y_0)$, thus establishing~\eqref{eq-NBRW-L2-dist-precise} for the  transitive case.
\end{remark}

\subsection{Extensions}\label{subsec:extensions}
We conclude with corollaries of the proof of Theorem~\ref{thm-nbrw}.

\subsubsection{Bipartite Ramanujan graphs} Following is the analog for \NBRW\ in the bipartite case; its \SRW\ counterpart follows from the cover-tree reduction.
\begin{corollary}
  \label{cor-nbrw-bipartite}
 Let $G=(V_0,V_1,E)$ be a bipartite Ramanujan graph on $n$ vertices with degree $d\geq 3$. Let
 $\mu_{t}$ be the $t$-step transition kernel of the \emph{\NBRW}, and let $\pi_0$ and $\pi_1$ be the uniform distribution on the $N/2$ directed edges originating from $V_0$ and $V_1$, respectively. Then for some fixed $c(d)>0$,
 \[ \max_{\substack{(x_0,y_0)\in\vE \\ x_0\in V_0} } \Big\| \frac{\mu_t\big((x_0,y_0),\cdot\big)}{\pi_{(t\bmod 2)}} - 1 \Big\|^2_{L^2(\pi_{(t\bmod 2)})} \leq \frac{c(d)}{\log n} \]
 at time
 \[t = \left\lceil \log_{d-1} n + 3\log_{d-1}\log n\right\rceil\,.\]
Consequently, on any sequence of such graphs, the \emph{\NBRW} that is modified to be lazy in its first step exhibits $L^1$-cutoff and $L^2$-cutoff at time $\log_{d-1} n$.
\end{corollary}
\begin{proof}
Following the arguments used to prove Theorem~\ref{thm-nbrw}, observe that in computing
$\E\big[| \mu_t\big((x_0,y_0),(x,y)\big)/\pi_{(t\bmod2)} - 1 |^2\big]$, the identity~\eqref{eq-l2-dist} becomes valid once we replace $N$ by $N/2$.
 The only other modification needed is to treat $\lambda_n = -d$, which produces the eigenvalue $\theta_n = -(d-1)$. Since all the coordinates of $w_n$ are $\pm N^{-1/2}$, the contribution of this eigenvalue to the right-hand of~\eqref{eq-mut-spectral1} is $1/N$, exactly that of the eigenvalue $d-1$ of $B$.
 The combined $2/N$ cancels via the modified identity~\eqref{eq-l2-dist}, thus~\eqref{eq-d2(t)-bound} becomes
 \[ \bigg\| \frac{\mu_t\big((x_0,y_0),\cdot\big)}{\pi_{(t\bmod2)}}-1\bigg\|^2_{L^2(\pi_{(t\bmod2)})} \leq  N (d-1)^{-t} \left(4(d-1)t^2 + 1\right) \,,\]
 which is $O(1/\log n)$ at the same value of $t$.
\end{proof}

\begin{corollary}
  \label{cor-srw-bipartite}
 Let $G=(V_0,V_1,E)$ be a bipartite Ramanujan graph on $n$ vertices with degree $d\geq 3$. Let
 $P^{t}$ be the $t$-step transition kernel of the \emph{\SRW}, and let $\pi_0$ and $\pi_1$ be the uniform distribution on $V_0$ and $V_1$, respectively. Let
Then for every fixed $s\in \R$ and every initial vertex $x$, the \emph{\SRW} at time
\begin{equation*}
  t = \tfrac{d}{d-2}\log_{d-1} n + s\sqrt{\log_{d-1} n}\,.
\end{equation*}
satisfies
\begin{equation*}
  \max_{x_0\in V_0}  \left\| P^t\big(x_0,\cdot\big)- \pi_{(t\bmod 2)}\right\|_{\tv} \to \P\left( Z > c_d\, s\right) \qquad\mbox{ as $n\to\infty$} \,,
\end{equation*}
where $Z$ is a standard normal random variable and $c_d = \tfrac{(d-2)^{3/2}}{2\sqrt{d(d-1)}}$.

Consequently, on any sequence of such graphs, the \emph{\SRW} that is modified to be lazy in its first step exhibits $L^1$-cutoff and $L^2$-cutoff at time $\frac{d}{d-2}\log_{d-1} n$.
\end{corollary}

\subsubsection{Weakly Ramanujan graphs}
It suffices to establish the result for the \NBRW\ (here we do not specify $D_\tv$ for the \SRW\ within the cutoff window, thus there is no  need to control the \NBRW\ within a window of $o(\sqrt{\log n})$),
and Theorem~\ref{thm-weakly} and Corollary~\ref{cor-dist-weakly} will then follow using the above reduction.
\begin{corollary}
\label{cor-nbrw-weakly}
 Fix $d\geq3$ and let $G$ be a $d$-regular graph on $n$ vertices whose nontrivial eigenvalues $\{\lambda_i\}_{i=2}^n$ all satisfy $|\lambda_i|\leq (1+\delta_n) 2\sqrt{d-1}$ for some $\delta_n$ going to 0 as $n\to\infty$.
Let  $\mu_{t}$ be the $t$-step transition kernel of the \emph{\NBRW}, and let $\pi$ be the uniform distribution on $\vE$.  For some fixed $c(d)>0$,
 \[ \max_{(x,y)\in\vE} \Big\| \frac{\mu_t\big((x,y),\cdot\big)}{\pi} - 1 \Big\|^2_{L^2(\pi)} \leq \frac{c(d)}{\log n}\]
 at time
 \[ t = \left\lceil \big(1+5\sqrt{\delta_n}\big)\log_{d-1}n + 3\log_{d-1}\log n\right\rceil\,. \]
  Consequently, on any sequence of such graphs, the \emph{\NBRW} exhibits $L^1$-cutoff and $L^2$-cutoff both at time $\log_{d-1} n$.
\end{corollary}
\begin{proof}
The analysis of blocks of $\Lambda$ corresponding to eigenvalues $\lambda_i$ ($i\geq 2$) of $A$ such that $|\lambda_i|\leq 2\sqrt{d-1}$ remains valid unchanged, and it remains to consider the effect of
 \begin{equation}
   \label{eq-lambda-epsilon-def}
   |\lambda_i| = (1+\epsilon) 2\sqrt{d-1} \quad\mbox{ for some }\quad 0 < \epsilon \leq \delta_n\,.
 \end{equation}
As mentioned in the proof of Theorem~\ref{thm-nbrw}, the fact that $G$ is Ramanujan is exploited when replacing $|\theta_i|^2t$ by $(d-1)^t$ for all $i\geq 2$ in the spectral decomposition~\eqref{eq-mut-spectral1}, and once again (just above~\eqref{eq-d2(t)-bound}) in the bound~\eqref{eq-gamma-bound} on $\gamma_i(t)$.
 For $\lambda_i$ as in~\eqref{eq-lambda-epsilon-def}, the corresponding real eigenvalues $\theta_i, \theta_i'$ of $B$ are given, as per~\eqref{eq-theta-lambda-quadratic}, by
\[ \left(1+\epsilon \pm \sqrt{\epsilon(2 + \epsilon)}\right)\sqrt{d-1} \,;\]
in particular, denoting $|\theta_i|>|\theta_i'|$, we have $|\theta_i| = (1+\sqrt{2\epsilon}+O(\epsilon))\sqrt{d-1}$ (while at the same time $|\theta_i'| < \sqrt{d-1}$).
We account for this modified value of $|\theta_i|^{2t}$ in the spectral decomposition of $\|\mu_t\big((x,y),\cdot\big)\|^2$ via the pre-factor
\[ \left(1+\sqrt{2\epsilon}+O(\epsilon)\right)^{2t} \leq \exp\left[\big(2\sqrt{2\delta_n} + O(\delta_n)\big)t\right]\,,\]
thus replacing the right-hand of~\eqref{eq-d2(t)-bound} by
\[ 2 N (d-1)^{-t}e^{\left[2\sqrt{2\delta_n}+O(\delta_n)\right]t}\left(4(d-1) t^2 + 1\right)\,.\]
For the designated value of $t$ (in which there is an extra additive term of $5\sqrt{\delta_n}\log_{d-1}n$ compared to $t$ from Theorem~\ref{thm-nbrw}) and using that $\delta_n \to 0$, we find that there exists some fixed $c(d)>0$ such that $\| \mu_t\big((x,y),\cdot\big)/\pi - 1 \|^2_{L^2(\pi)}$ is at most
\[ \frac{c(d) + o(1)}{\log n} \exp\left[ \big(2\sqrt2 -5\log(d-1) + o(1)\big) \sqrt{\delta_n} t\right]\,,\]
which is $O(1/\log n)$ since $2\sqrt2 < 5\log(d-1)$ for all $d\geq 3$.
\end{proof}
\begin{remark}
Suppose that, for some $\delta_n = o(1)$ and fixed $\epsilon'>0$, the graph $G$ has $|\lambda| \leq 2\sqrt{d-1} + \delta_n$
for all eigenvalues $\lambda$ except for $n^{o(1)}$ exceptional ones, which instead satisfy $|\lambda| < d-\epsilon'$. Each eigenvalue of the latter form corresponds to an additive term of $O(a^{2t})$ in the right-hand of~\eqref{eq-d2(t)-bound}, where $0<a<1$ depends only on $d$ and $\epsilon'$. For the prescribed $t$ from Corollary~\ref{cor-nbrw-weakly}, this
amounts to $O(n^{-\epsilon''})$ for some fixed $\epsilon''>0$, thus the overall contribution of these $n^{o(1)}$ exceptional eigenvalues is negligible and the same result holds.
\end{remark}

\section{Pinpointing the spectral decomposition}\label{sec:off-diag-spectrum}
The following proposition gives the precise moduli of the off-diagonal terms in $\Lambda$ from the spectral decomposition~\eqref{eq-B-decomposition} in Proposition~\ref{prop:spectral}.
\begin{proposition}\label{prop-alphai}
In the setting of Proposition~\ref{prop:spectral}, for all $i\geq 2$ we have
$\alpha_i=0$ if $\lambda_i=-d$ (and $i=n$), and otherwise
\begin{align*}
  |\alpha_i| &= \begin{cases} d-2&\qquad\mbox{if }|\lambda_i| \leq 2\sqrt{d-1}\,,\\
  \sqrt{d^2-\lambda_i^2} &\qquad\mbox{if }|\lambda_i|>2\sqrt{d-1}\,.
  \end{cases}
\end{align*}
\end{proposition}
\begin{proof}
Let $i\geq 2$, and for simplicity, omit its indices from the corresponding subscripts; namely,
let $\theta,\theta'$ correspond to the eigenvalue $\lambda\neq \pm d$, and let $f$ be so that $A f = \lambda f$ and $\|f\|=1$, where $A$ is the adjacency matrix of $G$.

\noindent\textbf{Case (1): $|\lambda|\neq 2\sqrt{d-1}$:}
Recalling $T_\theta$ from~\eqref{eq-T-def}, we claim that
\begin{equation}
   \label{eq-alpha-def}
   \alpha = \frac{\beta (\theta'-\theta)}{\sqrt{1-|\beta|^2}}\quad\mbox{ where }\quad \beta := \frac{\big<T_{\theta'}f,\, T_{\theta}f\big>}{\|T_{\theta'}f\| \,\|T_{\theta}f\|}\,.
 \end{equation}
Indeed, taking
\[ w =\frac{T_{\theta} f}{\|T_{\theta}f\|} \,,\quad w' =  \frac{T_{\theta'} f}{\|T_{\theta'}f\|} \,,\quad w'' = \frac{w'-\beta w}{\|w'-\beta w\|}\]
for $\beta$ as above gives $B w = \theta w$, $B w'=\theta' w$, and
\begin{equation}
  \label{eq-w''-normalizer}
  \|w'-\beta w\|^2 = 1 + |\beta|^2 - \beta\left<w,w'\right>-{\bar\beta}\overline{\left<w',w\right>} = 1 - |\beta|^2\,,
\end{equation}
so $w'' = (1-|\beta|^2)^{-1/2}(w'-\beta w)$ satisfies
\[ B w'' = \frac{\theta' w' - \beta\theta w}{\sqrt{1-|\beta|^2}} = \theta' w'' + \frac{\beta(\theta'-\theta)}{\sqrt{1-|\beta|^2}} = \theta' w'' + \alpha w\,,\]
as claimed. To estimate $\alpha$, observe that for every $f\in\ell^2(V)$,
\begin{equation}
  \label{eq-th-thbar}
  \|T_\theta f\|_{\ell^2(\vE)}^2 = d(|\theta|^2 +1)\|f\|_{\ell^2(V)}^2 - \left(\theta+\bar\theta\right)\left<Af, f\right>_{\ell^2(V)}\,.
\end{equation}

\noindent\textbf{Case (1.a): below the Ramanujan threshold}.
When $|\lambda| < 2\sqrt{d-1}$ we have $\theta'=\bar\theta \in \C\setminus \R$.
Since $\theta^2 + d-1=\lambda\theta$ and $|\theta|=\sqrt{d-1}$,
\begin{align*}
\|T_\theta f\|^2 &=
 d^2 - (\theta+\bar\theta)\lambda = d^2 - 2(d-1)-(\theta^2 + \bar\theta^2)  \\ &= d^2 - 2(d-1)\left[1+\cos(2\varphi)\right]
 =
 (d-2)^2 + 2(d-1)(1-\cos(2\varphi))\,,
\end{align*}
where we let $\theta = \sqrt{d-1} \exp(i\varphi)$. Similarly,
\begin{align*}
   \beta &= \frac{d (\bar\theta^2 + 1) - 2\bar\theta\lambda}{(d-2)^2+ 2(d-1)(1-\cos(2\varphi))}=\frac{(d-2)(\bar\theta^2-1)}{(d-2)^2 + 2(d-1)(1-\cos(2\varphi))}\,,
\end{align*}
and so
\begin{align*} 1-|\beta|^2 &=
1-\frac{(d-1)^2 + 1 - 2(d-1)\cos(2\varphi)}{\left[d-2 + 2\frac{d-1}{d-2}(1-\cos(2\varphi))\right]^2}\\
&=
\frac{4(\frac{d-1}{d-2})^2(1-\cos(2\varphi))^2 +2(d-1)(1-\cos(2\varphi))}{\left[d-2 + 2\frac{d-1}{d-2}(1-\cos(2\varphi))\right]^2}\,.
\end{align*}
Substituting $\cos(2\varphi)=1-2\sin^2\varphi$ we see that
\[ 1-|\beta|^2 = \frac{4(d-1)\left(1+4\frac{d-1}{(d-2)^2}\sin^2\varphi\right)\sin^2\varphi}
{\left(d-2 + 4\frac{d-1}{d-2}\sin^2\varphi\right)^2}= \frac{4(d-1)\sin^2\varphi}{(d-2)^2+4(d-1)\sin^2\varphi}\,,\]
and so
\[ \frac{|\beta|^2}{1-|\beta|^2} = \frac{1}{1-|\beta|^2}-1 = \frac{(d-2)^2}{4(d-1)\sin^2\varphi}\,.\]
Since $\theta-\theta' = 2\sqrt{d-1}\sin\varphi$, we conclude from~\eqref{eq-alpha-def} that $|\alpha| = d-2$.

\noindent\textbf{Case (1.b): above the Ramanujan threshold}.
 For $2\sqrt{d-1}<|\lambda|<d$, we have $\theta\neq\theta'\in\R$, and assume w.l.o.g.\ that $\theta > \theta'$. By~\eqref{eq-th-thbar} we get
\[ \|T_{\theta} f\|^2 = d(\theta^2+1)-2\theta\lambda = d(\theta^2+1)-2(\theta^2+d-1)= (d-2)(\theta^2-1)\,,\]
and for the same reason, $\|T_{\theta'} f\|^2 = (d-2)({\theta'}^2-1)$. Similarly,
\[ \left<T_\theta f, T_{\theta'}f\right> = d(\theta \theta' + 1) - (\theta + \theta')\lambda = d^2 - \lambda^2\,,\]
using that $\theta \theta' = d-1$ whereas $\theta+\theta' = \lambda$ through their definition in~\eqref{eq-theta-lambda-quadratic}. Since we also have $\theta^2 + {\theta'}^2 = \lambda^2 - 2(d-1)$, we see that
\[ (\theta^2 - 1)({\theta'}^2-1) = (d-1)^2 -\big(\lambda^2 - 2(d-1)\big) + 1 = d^2 - \lambda^2\,,\]
and altogether deduce that
\begin{align*}
  \beta &= \frac{d^2 - \lambda^2}{\big[(d-2)(\theta^2-1)\big]^{\frac12}\big[(d-2)({\theta'}^2-1)\big]^{\frac12}}=\frac{\sqrt{d^2-\lambda^2}}{d-2}\,.
\end{align*}
Therefore,
\[ \frac{\beta^2}{1-\beta^2} = \frac{1}{1-\beta^2}-1 = \frac{d^2-\lambda^2}{(d-2)^2-(d^2-\lambda^2)} = \frac{d^2-\lambda^2}{\lambda^2-4(d-1)}\,,\]
Recalling the definition~\eqref{eq-alpha-def} of $\alpha$, and using that $\theta - \theta' = \sqrt{\lambda^2 - 4(d-1)}$, we infer that
$  \alpha^2 = d^2-\lambda^2$.

\noindent\textbf{Case (2): at the Ramanujan threshold:} For $|\lambda|= 2\sqrt{d-1}$  we claim
\begin{equation}
  \label{eq-alpha-2root(d-1)}
\alpha = \frac{\|T_\theta f\|}{\|T_{1+\theta}f\|\sqrt{1-|\beta|^2}}\quad \mbox{ where }\quad \beta := \frac{\big<T_{1+\theta}f,\, T_{\theta}f\big>}{\|T_{1+\theta}f\| \,\|T_{\theta}f\|}\,.
\end{equation}
To see this, take
\[ w =\frac{T_{\theta} f}{\|T_{\theta}f\|} \,,\quad w' =  \frac{T_{1+\theta} f}{\|T_{1+\theta}f\|} \,,\quad w'' = \frac{w'-\beta w}{\|w'-\beta w\|}\,;\]
since $B w' = \theta w' + (\|T_\theta f\|/\|T_{1+\theta}f\|)w$ by~\eqref{eq-jordan}, while $\|w'-\beta w\|^2=1-|\beta|^2$ (by the same calculation as in~\eqref{eq-w''-normalizer}),
\[ B w'' = \frac{\theta w' + \|T_{1+\theta}f\|^{-1}w -\beta \theta w}{\sqrt{1-|\beta|^2}} = \theta w'' + \alpha w\,,\]
as claimed. To compute $\alpha$, we recall that $\theta=\lambda/2$, and infer from~\eqref{eq-th-thbar} that
\begin{align}
\|T_\theta f\|^2 &= d(\theta^2+1)-2\theta\lambda = d^2 - 2\theta \lambda = (d-2)^2\,,\label{eq-Tthetaf-rama}\\
\|T_{1+\theta}f\|^2 &= d \left((1+\theta)^2+1\right) - 2(1+\theta)\lambda=
(d-2)(d+2\theta-2)+d\,,\nonumber
\end{align}
as well as that
\[ \left<T_{1+\theta}f,T_\theta f\right> = d\big(\theta(1+\theta)+1\big) - (2\theta+1)\lambda = (d-2)(d+\theta-2)\,.\]
We therefore have
\begin{align*}
  1-\beta^2 &= 1- \frac{(d+\theta-2)^2}{(d-2)(d+2\theta-2)+d} = \frac{(d+\theta-2)(-\theta) + \theta(d-2)+d}{(d-2)(d+2\theta-2)+d}\\
&=  \frac{d-\theta^2}{(d-2)(d+2\theta-2)+d}=\frac{1}{\|T_{1+\theta}f\|^2}
\end{align*}
and so, by~\eqref{eq-alpha-2root(d-1)}--\eqref{eq-Tthetaf-rama}, $\alpha = \|T_\theta f\| = d-2$.
\end{proof}

\subsection*{Acknowledgements}
We thank Shayan Oveis Gharan for suggesting that we study cutoff on Ramanujan graphs, and Perla Sousi for comments on an earlier version of this manuscript. The research of E.L.\ was supported in part by NSF grant DMS-1513403.

\bibliographystyle{abbrv}
\bibliography{ramanujan_ref}

\begin{thebibliography}{10}

\bibitem{Aldous}
D.~Aldous.
\newblock Random walks on finite groups and rapidly mixing {M}arkov chains.
\newblock volume 986 of {\em Lecture Notes in Math.}, pages 243--297. Springer,
  Berlin, 1983.

\bibitem{AD}
D.~Aldous and P.~Diaconis.
\newblock Shuffling cards and stopping times.
\newblock {\em Amer. Math. Monthly}, 93(5):333--348, 1986.

\bibitem{AF}
D.~Aldous and J.~A. Fill.
\newblock Reversible markov chains and random walks on graphs, 2002.
\newblock Available at {\tt
  http://www.stat.berkeley.edu/\string~aldous/RWG/book.html}.

\bibitem{Alon86}
N.~Alon.
\newblock Eigenvalues and expanders.
\newblock {\em Combinatorica}, 6(2):83--96, 1986.

\bibitem{ABLS07}
N.~Alon, I.~Benjamini, E.~Lubetzky, and S.~Sodin.
\newblock Non-backtracking random walks mix faster.
\newblock {\em Commun. Contemp. Math.}, 9(4):585--603, 2007.

\bibitem{AM85}
N.~Alon and V.~D. Milman.
\newblock {$\lambda_1,$} isoperimetric inequalities for graphs, and
  superconcentrators.
\newblock {\em J. Combin. Theory Ser. B}, 38(1):73--88, 1985.

\bibitem{AFH15}
O.~Angel, J.~Friedman, and S.~Hoory.
\newblock The non-backtracking spectrum of the universal cover of a graph.
\newblock {\em Trans. Amer. Math. Soc.}, 367(6):4287--4318, 2015.

\bibitem{Bass92}
H.~Bass.
\newblock The {I}hara-{S}elberg zeta function of a tree lattice.
\newblock {\em Internat. J. Math.}, 3(6):717--797, 1992.

\bibitem{Bordenave}
C.~Bordenave.
\newblock A new proof of {F}riedman's second eigenvalue {T}heorem and its
  extension to random lifts.
\newblock 2015.
\newblock Preprint, available at {\tt arXiv:1502.04482}.

\bibitem{ChenSaloffCoste08}
G.-Y. Chen and L.~Saloff-Coste.
\newblock The cutoff phenomenon for ergodic {M}arkov processes.
\newblock {\em Electron. J. Probab.}, 13:no. 3, 26--78, 2008.

\bibitem{Chung89}
F.~R.~K. Chung.
\newblock Diameters and eigenvalues.
\newblock {\em J. Amer. Math. Soc.}, 2(2):187--196, 1989.

\bibitem{CFM94}
F.~R.~K. Chung, V.~Faber, and T.~A. Manteuffel.
\newblock An upper bound on the diameter of a graph from eigenvalues associated
  with its {L}aplacian.
\newblock {\em SIAM J. Discrete Math.}, 7(3):443--457, 1994.

\bibitem{DSV03}
G.~Davidoff, P.~Sarnak, and A.~Valette.
\newblock {\em Elementary number theory, group theory, and {R}amanujan graphs},
  volume~55 of {\em London Mathematical Society Student Texts}.
\newblock Cambridge University Press, Cambridge, 2003.

\bibitem{DiaSha_81}
P.~Diaconis and M.~Shahshahani.
\newblock Generating a random permutation with random transpositions.
\newblock {\em Z. Wahrsch. Verw. Gebiete}, 57(2):159--179, 1981.

\bibitem{DurrettRGW}
R.~Durrett.
\newblock {\em Random graph dynamics}.
\newblock Cambridge Series in Statistical and Probabilistic Mathematics.
  Cambridge University Press, Cambridge, 2010.

\bibitem{FellerI}
W.~Feller.
\newblock {\em An introduction to probability theory and its applications.
  {V}ol. {I}}.
\newblock Third edition. John Wiley \& Sons, Inc., New York-London-Sydney,
  1968.

\bibitem{Friedman08}
J.~Friedman.
\newblock A proof of {A}lon's second eigenvalue conjecture and related
  problems.
\newblock {\em Mem. Amer. Math. Soc.}, 195(910):viii+100, 2008.

\bibitem{FK}
J.~Friedman and D.~Kohler.
\newblock The relativized second eigenvalue conjecture of {A}lon.
\newblock 2014.
\newblock Preprint, available at {\tt arXiv:1403.3462}.

\bibitem{HLW06}
S.~Hoory, N.~Linial, and A.~Wigderson.
\newblock Expander graphs and their applications.
\newblock {\em Bull. Amer. Math. Soc. (N.S.)}, 43(4):439--561 (electronic),
  2006.

\bibitem{KS00}
M.~Kotani and T.~Sunada.
\newblock Zeta functions of finite graphs.
\newblock {\em J. Math. Sci. Univ. Tokyo}, 7(1):7--25, 2000.

\bibitem{Lalley93}
S.~P. Lalley.
\newblock Finite range random walk on free groups and homogeneous trees.
\newblock {\em Ann. Probab.}, 21(4):2087--2130, 1993.

\bibitem{LS-gnd}
E.~Lubetzky and A.~Sly.
\newblock Cutoff phenomena for random walks on random regular graphs.
\newblock {\em Duke Math. J.}, 153(3):475--510, 2010.

\bibitem{LSexp}
E.~Lubetzky and A.~Sly.
\newblock Explicit expanders with cutoff phenomena.
\newblock {\em Electron. J. Probab.}, 16:no. 15, 419--435, 2011.

\bibitem{Lubotzky2010}
A.~Lubotzky.
\newblock {\em Discrete groups, expanding graphs and invariant measures}.
\newblock Modern Birkh\"auser Classics. Birkh\"auser Verlag, Basel, 2010.

\bibitem{LPS88}
A.~Lubotzky, R.~Phillips, and P.~Sarnak.
\newblock Ramanujan graphs.
\newblock {\em Combinatorica}, 8(3):261--277, 1988.

\bibitem{LP}
R.~Lyons and Y.~Peres.
\newblock {\em Probability on Trees and Networks}.
\newblock Cambridge University Press.
\newblock In preparation. Current version available at {\tt
  http://pages.iu.edu/\string~rdlyons/}.

\bibitem{MSS15}
A.~Marcus, D.~A. Spielman, and N.~Srivastava.
\newblock Interlacing families {I}: bipartite {R}amanujan graphs of all
  degrees.
\newblock {\em Ann. of Math.}, 182(1):307--325, 2015.

\bibitem{Margulis88}
G.~A. Margulis.
\newblock Explicit group-theoretic constructions of combinatorial schemes and
  their applications in the construction of expanders and concentrators.
\newblock {\em Problemy Peredachi Informatsii}, 24(1):51--60, 1988.

\bibitem{Nilli91}
A.~Nilli.
\newblock On the second eigenvalue of a graph.
\newblock {\em Discrete Math.}, 91(2):207--210, 1991.

\bibitem{PeresAIM04}
Y.~Peres.
\newblock American Institute of Mathematics (AIM) research workshop ``Sharp
  Thresholds for Mixing Times'', Palo Alto, December 2004.
\newblock Summary available at {\tt http://www.aimath.org/WWN/mixingtimes}.

\bibitem{Sardari}
N.~T. Sardari.
\newblock Diameter of {R}amanujan graphs and random {C}ayley graphs with
  numerics.
\newblock 2015.
\newblock Preprint, available at {\tt arXiv:1511.09340}.

\bibitem{Sarnak}
P.~Sarnak.
\newblock Letter to {S}cott {A}aronson and {A}ndrew {P}ollington on the
  {S}olovay--{K}itaev {T}heorem and {G}olden {G}ates (with an appendix on
  optimal lifting of integral points).
\newblock February 2015.
\newblock Available at {\tt http://publications.ias.edu/sarnak/paper/2637}.

\bibitem{Serre97}
J.-P. Serre.
\newblock R\'epartition asymptotique des valeurs propres de l'op\'erateur de
  {H}ecke {$T\sb p$}.
\newblock {\em J. Amer. Math. Soc.}, 10(1):75--102, 1997.

\bibitem{SteinWeiss71}
E.~M. Stein and G.~Weiss.
\newblock {\em Introduction to {F}ourier analysis on {E}uclidean spaces}.
\newblock Princeton University Press, Princeton, N.J., 1971.

\bibitem{Woess09}
W.~Woess.
\newblock {\em Denumerable {M}arkov chains}.
\newblock European Mathematical Society (EMS), Z\"urich, 2009.
\newblock Generating functions, boundary theory, random walks on trees.

\end{thebibliography}

\end{document}